\documentclass[a4paper,11pt,reqno]{amsart}
\usepackage{amsmath, amssymb}
\usepackage{relsize}
\usepackage[noadjust]{cite}
\usepackage{color}
\usepackage[dvipsnames]{xcolor}
\usepackage{mathtools}
\usepackage[normalem]{ulem}

\newcommand\redout{\bgroup\markoverwith
{\textcolor{red}{\rule[.5ex]{2pt}{0.4pt}}}\ULon}

\usepackage{hyperref, enumitem}
\hypersetup{
  colorlinks   = true, 
  urlcolor     = blue, 
  linkcolor    = Purple, 
  citecolor   = red 
}
\usepackage{amsmath,amsthm,amssymb}
\usepackage{cleveref}

\usepackage[margin=2.5cm]{geometry}
\usepackage{caption} 
\usepackage[ntheorem]{empheq}

\usepackage{tikz-cd}
\usetikzlibrary{calc,fit,matrix,arrows,automata,positioning}
\usetikzlibrary{decorations.pathreplacing,angles,quotes}
\usepackage{stmaryrd}

\usetikzlibrary{fit}
\tikzset{%
  highlight/.style={rectangle,rounded corners,fill=red!15,draw,
    fill opacity=0.5,thick,inner sep=0pt}
}

\usepackage{array}
\newcolumntype{x}[1]{>{\centering\arraybackslash\hspace{0pt}}p{#1}}

\theoremstyle{definition}
\newtheorem{theorem}{Theorem}[section]
\newtheorem{definition}[theorem]{{{Definition}}}
\newtheorem{example}[theorem]{{{Example}}}
\newtheorem{notation}[theorem]{{{Notation}}}
\newtheorem{remark}[theorem]{{{Remark}}}

\newtheorem{corollary}[theorem]{{{Corollary}}}
\newtheorem{proposition}[theorem]{{{Proposition}}}
\newtheorem{lemma}[theorem]{{{Lemma}}}

\newcommand{\C}{\mathcal C}

\newcommand{\mD}{\mathcal{D}}
\newcommand{\mM}{\mathcal{M}}
\newcommand{\mL}{\mathcal{L}}
\newcommand{\mU}{\mathcal{U}}
\newcommand{\mF}{\mathcal{F}}
\newcommand{\mO}{\mathcal{O}}
\newcommand{\mI}{\mathcal{I}}
\newcommand{\mZ}{\mathcal{Z}}
\newcommand{\mS}{\mathcal{S}}

\newcommand{\mN}{\mathcal{N}}
\newcommand{\mH}{\mathcal{H}}

\newcommand{\mP}{\mathcal{P}}
\newcommand{\F}{\mathbb F}
\newcommand{\R}{\mathbb R}
\newcommand{\la}{\langle}
\newcommand{\ra}{\rangle}

\newcommand{\rRP}{\bar{\mathcal{P}}_q^n}
\newcommand{\RP}{\mathcal{P}_q^n}

\newcommand\qqbin[2]{\left[\begin{matrix} #1 \\ #2 \end{matrix} \right]}

\DeclareMathOperator{\rk}{rk}

\DeclareMathOperator{\Cl}{Cl}
\DeclareMathOperator{\cl}{cl}

\DeclareMathOperator{\conv}{conv}
\DeclareMathOperator{\ver}{Vert}

\usepackage{scrextend}

\allowdisplaybreaks

\title{The polytope of all $q$-rank functions}
\usepackage[foot]{amsaddr}
\author[G. N. Alfarano]{Gianira N. Alfarano$^{1}$}
\address{$^1$\textnormal{Université de Rennes, IRMAR, Rennes, France.}}
 \email{gianira-nicoletta.alfarano@univ-rennes.fr}
 \author[S. Degen]{Sebastian Degen$^2$}
\address{$^2$\textnormal{Universit\"{a}t Bielefeld, Fakult\"{a}t f\"{u}r Mathematik, Bielefeld, Germany.}}
 \email{sdegen@math.uni-bielefeld.de}
\begin{document}
\begin{abstract}
    A $q$-rank function is a real-valued function defined on the subspace lattice that is non-negative, upper bounded by the dimension function, non-decreasing, and satisfies the submodularity law. Each such function corresponds to the rank function of a $q$-polymatroid. In this paper, we identify these functions with points in a polytope. We show that this polytope contains no interior lattice points, implying that the points corresponding to $q$-matroids are among its vertices. We investigate several properties of convex combinations of two lattice points in this polytope, particularly in terms of independence, flats, and cyclic flats. Special attention is given to convex combinations of paving and uniform $q$-matroids.
\end{abstract}

\maketitle

\section{Introduction}
The concept of matroid abstracts the notion of independence in vector spaces and graphs. Matroids have been extended in various directions, including polymatroids~\cite{edmonds2003submodular}, valuated matroids~\cite{dress1992valuated}, and, more recently, latroids~\cite{vertigan2004latroids, gorla2025latroids} and $q$-analogues such as $q$-matroids~\cite{jurrius2018defining}, $q$-polymatroids~\cite{gorla2019rank, shiromoto2019codes}, and $\mL$-polymatroid~\cite{alfarano2023critical, byrne2023cyclic}.

In the classical setting, let $S$ be a finite set of cardinality $n$. The rank function of a polymatroid is a real-valued function defined on all subsets of $S$ that is non-decreasing, submodular, and upper bounded by the cardinality function. The cone of all such functions has been studied in \cite{nguyen1978semimodular}, particularly to characterize its extremal rays and extremal polymatroids. Further developments appear in \cite{girlich1995cone, kashiwabara2000extremality}, and analogous investigations have been conducted for supermodular functions; see \cite{grabisch2019cone}.

In this paper, we focus on \emph{$q$-rank functions}; i.e., rank functions of $q$-polymatroids. Let $\mathbb{F}_q$ be the finite field of order $q$, let $E$ be an $n$-dimensional vector space over $\mathbb{F}_q$, and let $\mathcal{L}(E)$ be the lattice of subspaces of $E$ ordered by inclusion. A $q$-rank function is defined on $\mathcal{L}(E)$ and is real-valued, non-negative, upper-bounded by the dimension function, non-decreasing, and submodular. If, in particular, such a function is integer-valued, it is the rank function of a $q$-matroid. We define the \emph{polytope of all $q$-rank functions}, $\mathcal{P}_q^n \subseteq \mathbb{R}^{|\mathcal{L}(E)|}$, as the polytope defined by the linear inequalities given by the above properties. More specifically, by ordering the spaces in $\mathcal{L}(E)$, we see that a point $(p_X)_{X \in \mathcal{L}(E)} \in \mathcal{P}_q^n$ has coordinates indexed by the spaces in $\mathcal{L}(E)$, and the entries are precisely the values that a $q$-rank function takes on each $X \in \mathcal{L}(E)$.
We show that $\mathcal{P}_q^n$ has dimension $|\mL(E)| - 1$; see \Cref{prop:dimension}. However, as this dimension grows rapidly with $n$, analyzing its geometry becomes challenging. One of our main results in this direction is that the integer points (lattice points) are vertices of $\mathcal{P}_q^n$; see \Cref{thm:integer_vertices}. 
In particular, this implies that $\mathcal{P}_q^n$ contains no integer points in its interior. Moreover, $q$-matroids correspond precisely to points with integer coordinates. It should be noted, however, that $\mathcal{P}_q^n$ is not an integral polytope; besides these integer points, there exist other vertices whose coordinates are non-integer. Moreover, as soon as $n$ grows, the number of vertices also increases. Since this polytope is described by constraints with rational coefficients, it is a rational polytope, that is, all its vertices have rational coordinates. 

We introduce convex combinations of points in $\mathcal{P}_q^n$ as a natural operation on $q$-polymatroids. In addition to establishing general results on the flats of such combinations (see \Cref{subsec: flats}), we focus specifically on convex combinations of $q$-matroids with rational coefficients. We investigate how key structural elements, such as cyclic spaces and $\mu$-independent spaces (where $\mu$ is a denominator of the resulting $q$-polymatroid), behave under this operation. Specifically, we express these features of the resulting $q$-polymatroid in terms of the cyclic and independent spaces of the original $q$-matroids. We specialize our analysis to convex combinations involving two important families of $q$-matroids: paving $q$-matroids and uniform $q$-matroids. For paving $q$-matroids, we focus on a construction introduced in \cite{gluesing2022q}. In this context, we introduce the notion of \emph{$\mu$-pavingness} for $q$-polymatroids and show that the convex combination of two paving $q$-matroids, built using this special construction, yields a $\mu$-paving $q$-polymatroid, where the denominator $\mu$ is not necessarily principal; see \Cref{thm: mu_paving_conv_comb}. We further characterize important structural features of the resulting $q$-polymatroid: its flats (\Cref{thm: flats_paving_conv_comb}), cyclic spaces (\Cref{thm: cyclic_paving_conv_comb}), and cyclic flats (\Cref{coro: cyclic_flats_paving_conv_comb}). For the convex combination of two uniform $q$-matroids, we provide characterizations of $\mu$-independent spaces, under specific constraints (see \Cref{thm:mu-indep_uniform} and \Cref{thm: mu_indep_uniform2}), and cyclic flats (see \Cref{prop:cyc_flats_uniform}). These ideas are then extended to the convex combination of $n-2$ uniform $q$-matroids, yielding further results on $\mu$-independence (see \Cref{thm:mu_indep_flag_uniform}).

Another well-studied concept in matroid theory is the characteristic polynomial. Extensions of this invariant have been proposed for $q$-matroids, integer $q$-polymatroids, and $\mL$-matroids; see~\cite{jany2023proj_matroid, byrne2024weighted, alfarano2023critical}. In this work, we introduce a broader notion of the characteristic polynomial, tailored for $q$-polymatroids whose $q$-rank functions take rational values. Specifically, we define an invariant, the \emph{characteristic Puiseux polynomial}, associated with such $q$-polymatroids. This invariant is not a polynomial in the classical sense: it is a truncated Puiseux series, meaning a finite linear combination of monomials whose exponents may be rational (and possibly negative). Despite this departure from classical polynomials, we adopt the name characteristic Puiseux polynomial to emphasize its structural role and its connection to existing generalizations in matroid theory.
We show that the characteristic Puiseux polynomial of the convex combination with rational coefficients of two paving $q$-matroids, derived from the construction in \cite{gluesing2022q} can be expressed in terms of the characteristic polynomials of the constituent $q$-matroids. Notably, we show that the polynomial of the resulting $q$-polymatroid is determined by the characteristic polynomial of only one of the original $q$-matroids; see \Cref{prop: char_Puiseux_series_paving_conv_comb}. 

One of the drawbacks of the polytope of all $q$-rank functions is that we are not able to distinguish representable $q$-polymatroids. This difficulty is due to the lack of an axiomatic definition that applies only  for the representable case. However, we illustrate some examples of convex combinations of special representable $q$-polymatroids arising from rank-metric codes; see \Cref{thm: mu_indep_poly}. 

The rest of the paper is organized as follows. In \Cref{sec: prelims}, we provide the basic notions on $q$-polymatroids. In \Cref{sec: polytope}, we introduce the polytope of all $q$-rank functions and initiate the study of its geometric properties. In \Cref{sec: conv_combi}, we define the convex combination of $q$-polymatroids and provide its first properties. In \Cref{sec: paving_conx_comb}, we characterize the convex combination with rational coefficients of some paving $q$-matroids. In \Cref{sec: uniform_conv_comb}, we consider the convex combination of uniform $q$-matroids. In \Cref{sec: Puiseux_polyn}, we introduce the characteristic Puiseux  polynomial for rational $q$-rank functions and we study the characteristic Puiseux polynomial of the convex
combination of paving $q$-matroids. In \Cref{sec: rep_q_polymats}, we provide some examples of convex combinations of representable $q$-polymatroids. We conclude with some open questions and future research directions in \Cref{sec: final_remarks}.

\section{Preliminaries}\label{sec: prelims}

In this section, we outline some basic facts and definitions of polytopes and $q$-polymatroids.

{\subsection{Polytopes} In this subsection, we collect some background material about polytopes and their properties.
More details can be found in \cite{grunbaum1967convex,ziegler2012lectures}.

We start by recalling that a \textbf{polyhedron} is a set of the form $\mP=\{x \in \R^m \mid A  x^\top \le b^\top \}$, where $A \in \R^{\ell \times m}$, $\ell,m \ge 1$, $b \in \R^\ell$, and $\le$ is applied componentwise. In other words, a polyhedron~$\mP$ is a region defined by finitely many linear inequalities. The region $\mP$ is called a \textbf{polytope} if it is bounded. 
Since this definition implies that a polytope is the intersection of a finite number of half-spaces, such a description is called \textbf{half-space representation}, \textbf{$\mathcal{H}$-representation} or \textbf{$\mathcal{H}$-description}.
A fundamental result states that every polytope is the convex hull of a finite set of points. For a set $S \subseteq \R^m$, we let $\conv(S)$ denote its convex hull, where $\conv(\emptyset)=\emptyset$.

\begin{theorem}[see e.g.~\cite{grunbaum1967convex}]\label{thm:fundamental_poly}
Let $\mP \subseteq \R^m$ be a polytope. Then there exists a finite set $S \subseteq \R^m$ such that $\mP=\conv(S)$.
\end{theorem}
In particular, a polytope is a convex region in $\R^m$. Note that we follow the notation of standard texts like \cite{grunbaum1967convex, ziegler2012lectures}, which use the term ``polytope" for a bounded convex polyhedron, and the word ``polyhedron" for the more general, possibly unbounded object.

A \textbf{supporting hyperplane} of $\mP$ is a hyperplane $H = \{x \in \R^m \mid a x^\top = b\}$, where $a \in \R^m \setminus \{0\}$ and $b \in \R$, such that $\mP$ is entirely contained in one of the two closed half-spaces determined by~$H$ and $H \cap \mP \neq \emptyset$. A \textbf{face} of $\mP$ is any set of the form $F = \mP \cap H$, where $H$ is a supporting hyperplane of $\mP$. The \textbf{dimension} of a face $F = \mathcal{P} \cap H$ is its dimension as an affine subspace.

A \textbf{vertex} of a polytope $\mP \subseteq \R^m$ is an element $v \in \mP$ with $v \notin \conv(\mP \setminus \{v\})$. It is a standard result that a point $v$ is a vertex if and only if it is a $0$-dimensional face, meaning that it is the unique intersection of $\mathcal{P}$ with some supporting hyperplane.
The set of vertices of $\mP$ is denoted by $\ver(\mP)$. Note that if $\mP=\conv(S)$ is a polytope then $\ver(\mP) \subseteq S$.
Thus $\mP=\conv(\ver(\mP))$. Moreover, a nonempty polytope has at least one vertex. 
The description of a polytope as the convex hull of a finite set  is called \textbf{$\mathcal{V}$-representation} or \textbf{$\mathcal{V}$-description}.

\subsection{\emph{q}-Polymatroids} We recall the necessary background on $q$-polymatroids.
Throughout this paper we will make use of the following notation.

\begin{notation}\label{not:standing_notation}
Let $q$ be a prime power and let $\F_q$ be the finite field with $q$ elements. Let $n\geq 2$ be a positive integer and $E=\F_q^n$. For a space $X$, we denote by $\mL(X)$ the collection of subspaces of $X$. For $0\leq i\leq \dim(X)$, we denote by $\mL(X)_{\leq i}$ and $\mL(X)_{\geq i}$ the collections of all subspaces of~$X$ with dimension at most $i$ and at least $i$, respectively, while we denote by $\mL(X)_{i}$ the collection of all $i$-dimensional subspaces of $X$. Further, we denote by $\mL(X)^\ast$ the collection $\mL(X)\setminus\langle 0 \rangle$. For ease of notation, given a space $X$, we denote by $\mH(X)$ the set of codimension-one subspaces of $X$.
We write $X\leq Y$ to indicate that $X$ is a subspace of $Y$. If a subspace is to be understood as being one-dimensional, we represent it by a lowercase letter, so for instance, $x \leq X$ means that $x$ is a one-dimensional subspace of $X$. 
Finally, we set $T:=|\mL(E)|$. 
\end{notation}

The first results on $q$-polymatroids and their relation to coding theory can be read in \cite{gorla2019rank} and~\cite{shiromoto2019codes}. The following definition is from \cite[Def. 4.1]{gorla2019rank}.

\begin{definition}\label{def:rank_axioms}
A function $\rho: \mL(E)\longrightarrow \mathbb{R}$ is a \textbf{$q$-rank function} if it satisfies the following axioms.
\begin{itemize}
	\item[(R1)] \emph{Boundedness}: $0\leq \rho(A) \leq \dim(A)$, for all $A \in \mL(E)$.
	\item[(R2)] \emph{Monotonicity}: $A\leq B \Rightarrow \rho(A)\leq \rho(B)$,  for all $A,B \in \mL(E)$.
	\item[(R3)] \emph{Submodularity}: $\rho(A \cap B)+\rho(A+B)\leq \rho(A) +\rho(B)$, for all $A,B \in \mL(E)$.
\end{itemize}
A \textbf{$q$-polymatroid} is a pair $\mathcal{M}=(\mL(E),\rho)$ for which $\rho$ is a $q$-rank function.
\end{definition}
The value $\rho(E)$ is called the \textbf{rank of $\mathcal{M}$} and $E$ is called the \textbf{ground space} of $\mM$.

The definition of $q$-polymatroid appeared in different forms in the literature. In \cite[Def. 2.1]{gluesing2022independent} the rank function is defined as in \Cref{def:rank_axioms}, with the sole difference that it takes only rational values; in this case we talk about \textbf{rational $q$-polymatroids}. Fundamental concepts regarding rational $q$-polymatroids have been investigated in \cite{gluesing2022q} and \cite{gluesing2022independent}.
A number $\mu\in\mathbb{Q}_{\geq 0}$ is called a \textbf{denominator of $\rho$}  (and $\mathcal{M}$) if $\mu \rho(X) \in \mathbb{N}_0$ for all $X\in\mL(E)$. The smallest denominator is called the \textbf{principal denominator}. A $q$-polymatroid with
principal denominator $1$ (i.e., $\rho(X)\in\mathbb{N}_0$ for all $X$) is a $q$-matroid.  

Next, for any $r\in\mathbb{N}$ a
$(q, r)$-polymatroid as in \cite[Def. 2]{shiromoto2019codes} by Shiromoto, \cite[Def. 1]{ghorpade2020polymatroid} 
and \cite[Def. 1]{byrne2024weighted} can be turned into a rational $q$-polymatroid with denominator $r$ by
dividing the rank function by $r$. Conversely, given a rational $q$-polymatroid with denominator $\mu$,
then $(\mL(E), \mu\rho)$ is a $(q,\lceil \mu\rceil)$-polymatroid. Finally, in \cite[Def. 2.12]{alfarano2023critical} and \cite[Def. 2.2]{byrne2023cyclic} an $\mL$-polymatroid is defined as a direct extension of the previous definitions to arbitrary finite complemented modular lattices.

We will use the following definition of independent spaces proposed in \cite[Def. 3.1]{gluesing2022independent}. 

\begin{definition}\label{def:indep_spaces}
Let $\mM = (\mL(E), \rho)$ be a rational $q$-polymatroid with denominator $\mu$, not necessarily principal.  A space $I\in\mL(E)$ is called \textbf{$\mu$-independent} if
$$\rho(J)\geq \frac{\dim(J)}{\mu},$$ 
for all subspaces $J\leq I$.
$I$ is called \textbf{$\mu$-dependent} if it is not $\mu$-independent. A \textbf{$\mu$-circuit} is a $\mu$-dependent space for which all proper subspaces are $\mu$-independent. A $1$-dimensional $\mu$-dependent space is called a \textbf{$\mu$-loop}. We define $\mathcal{I}_\mu=\mathcal{I}_\mu(\mathcal{M}) = \{I \in\mL(E) : I \textnormal{ is } \mu\textnormal{-independent}\}$. 
A $\mu$-independent space $I$ for which $\rho(I)=\dim(I)$ is called a \textbf{strong independent} space.
\end{definition}

In \cite[Def. 3.8]{gluesing2022independent}, the authors define the $\mu$-bases of a space $V\in\mL(E)$ as the inclusion-maximal $\mu$-independent subspaces of $V$. Furthermore, they show that all the $\mu$-bases of $V$ have the same rank value, which is equal to the rank of $V$.

In case of $\mM$ being a $q$-matroid, the notion of strong independence coincides with the classical one. Specifically, a subspace $A \in \mathcal{L}(E)$ is called an independent space of the $q$-matroid $\mathcal{M}$ if $\rho(A) = \dim A$. If $X \in \mL(E)$ and $\rho(X) = 0$, then $X$ is called a \textbf{loop space}. A $1$-dimensional loop space is called a \textbf{loop} of $\mathcal{M}$. By submodularity, it follows that if $x, y \in\mL(E)$ are loops, their vector space sum $x+y$ must also be a loop space. Consequently, there exists a unique maximal space containing all loops; we call this \textbf{the loop space} of $\mathcal{M}$.
A subspace of $E$ that is not independent is called {\bf dependent} of $\mM$. We call $C \in \mL(E)$ a {\bf circuit} if it is itself a dependent space and every proper subspace of $C$ is independent. 

We finally define flats and cyclic spaces of a $q$-polymatroid $\mM=(\mL(E),\rho)$. 

\begin{definition}\label{def:flats}
    Let $X\in\mL(E)$. We say that $X$ is a \textbf{flat} of $\mM$ if $\rho(X)<\rho(X+ v)$ for all one-dimensional spaces $ v\in\mL(E)_1\setminus \mL(X)_1$.
\end{definition}

The following definition comes from \cite[Def. 4.2]{byrne2023cyclic}.

\begin{definition}\label{def:cyc_spaces}
Let $X\in\mL(E)$. We say that $X$ is \textbf{cyclic} if for all $H\in\mH(X)$ one of the following is satisfied:
\begin{enumerate}
    \item $\rho(X)=\rho(H)$ or
    \item $0 < \rho(X) - \rho(H)$ and there exists $ a\in \mL(X)_1\setminus \mL(H)_1$ such that $\rho(X) - \rho(H) < \rho(a)$.
\end{enumerate}
\end{definition}

Furthermore, we call a space $X\in\mL(E)$ a \textbf{cyclic flat} of $\mM$ if it is both a flat and a cyclic space of $\mM$.  We denote by $\mF(\mM), \; \mO(\mM)$ and $\mZ(\mM)$ the collections of flats, cyclic spaces and cyclic flats of $\mM$. Note that for $q$-matroids the notion of flats is identical, while, for cyclic spaces, the condition~(2) is superfluous.

Next we recall two well-known families of $q$-matroids, namely \emph{uniform $q$-matroids} (\cite{jurrius2018defining}) and \emph{paving $q$-matroids} (\cite{luerssen2023representabilitydirectsum}), which we will use in the next sections.

\begin{definition}\label{def:uniform}
Let $0\leq k\leq n$. For $V\in\mL(E)$, let $\rho(V) :=\min\{k,\dim(V)\}$. Then $(\mL(E),\rho)$ is a $q$-matroid, called the \textbf{uniform $q$-matroid on $E$ of rank $k$}, and denoted by~$\mU_{k,n}(q)$.
\end{definition}

\begin{definition}
    A $q$-matroid $(\mL(E),\rho)$ is said to be \textbf{paving} if for every  circuit $C$ we have $\dim(C)\geq\rho(E)$.
\end{definition}


\section{The polytope of all \emph{q}-rank functions}\label{sec: polytope}

In this section, we introduce the polytope of all $q$-rank functions and we describe some of its properties. In order to do this, we endow $\mL(E)$ with a linear order in the following way
\[
   \la 0\ra \preceq \mL(E)_1 \preceq \mL(E)_2 \preceq \ldots \preceq \mL(E)_{n-1} \preceq E,  
\]
where $\mL(E)_i$ is also endowed with a linear order for each $1\leq i\leq n-1 $. 
For every $0\leq\ell\leq n$, the number of $\ell$-dimensional subspaces of $E$ is denoted by the Gaussian binomial coefficient
$$\qqbin{n}{\ell}_q=\frac{(q^n-1)(q^{n-1}-1)\cdots(q^{n-\ell+1}-1)}{(q-1)(q^2-1)\cdots(q^\ell-1)}.$$
Hence, the size of $\mL(E)$ is $T=\sum\limits_{i=0}^n\qqbin{n}{i}_q$.

Now we consider the vector space $\mathbb{R}^T$, which contains vectors of the form $v=(v_X)_{X\in\mL(E)}$. 

\begin{definition}\label{def:rank_poly}
The \textbf{$q$-rank polyhedron on $E$} is the region
\begin{align*}
    \mathcal{P}_q^n:=&\{v=(v_X)_{X\in\mL(E)}\in\mathbb{R}_{\geq 0}^T\;\mid\; v \text{ satisfies the constraints } (*)\},
\end{align*}
where the constraints $(*)$ are the following:
\begin{enumerate}
    \item $v_X \leq \dim(X)$, for all $X\in\mL(E)$.
    \item $v_X-v_Y \leq 0$, for all $X,Y\in\mL(E)$ such that $X\leq Y$.
    \item $v_{X+Y}+v_{X\cap Y}-v_{X}-v_{Y} \leq 0$, for all $X,Y\in\mL(E)$.
 \end{enumerate}
We refer to these inequalities as of type 1, type 2 or type 3. 
\end{definition}
Note that the inequalities of type 1, 2 and 3 are exactly the axioms (R1), (R2), (R3) of a $q$-rank function. 
Hence, every $v\in\mP_q^n$ is the image vector of a $q$-rank function.
The first rank axiom ensures that every coordinate of $v$ is non-negative and bounded from above. This implies that the polyhedron $\mathcal{P}_q^n$  from  \Cref{def:rank_poly} is a polytope  in $\mathbb{R}^T$ that lies completely in $\mathbb{R}_{\geq 0}^T$. We call it the \textbf{polytope of all $q$-rank functions on $E$} or \textbf{$q$-rank polytope on $E$}.

\begin{remark}
The $q$-rank polytope $\mathcal{P}_q^n$ contains all the possible $q$-rank functions. The points of~$\mathcal{P}_q^n$ with integer coordinates (lattice points) correspond to $q$-matroid rank functions. Note that the term ``lattice points" here refers to integer-coordinate points only and should not be confused with covers of the minimal elements of a lattice in the order-theoretic sense.
\end{remark}

\begin{example}\label{ex:first_ex}
Consider $E=\F_2^2$ and $\mL(E)=\{\la 0 \ra \preceq \la 0 1\ra \preceq \la 1 0\ra \preceq \la 1 1\ra \preceq E\}$. Then the constraints describing $\mathcal{P}_2^2$ are given as follows:
    
\begin{center}
    \begin{tabular}{c|c|c}
        \textbf{Type $1$} & \textbf{Type $2$} & \textbf{Type $3$}\\ \hline
             $v_1=0$ & $v_2-v_5 \leq 0$  & $v_5-v_2-v_3 \leq 0$ \\
             $v_2 \leq 1$ & $v_3-v_5 \leq 0$ & $v_5-v_2-v_4 \leq 0$ \\
             $v_3 \leq 1$ & $v_4-v_5 \leq 0$ & $v_5-v_3-v_4 \leq 0$ \\
             $v_4 \leq 1$ &  & \\
             $v_5\leq 2$ &  &
    \end{tabular}
\end{center}

Note that we omitted all inequalities of type $2$ and $3$ involving pairs of subspaces $(\la0 \ra, X)$ and also all inequalities of type $3$ involving pairs of subspaces $(X,E)$, since they are redundant. Using these inequalities we can determine the $\mathcal{H}$-representation of $\mathcal{P}_2^2\subseteq\mathbb{R}^5$, given by $Av^\top\leq b^\top$, for
    \[
        A = \begin{pmatrix}
             1&0&0&0&0\\
             0&1&0&0&0\\
             0&0&1&0&0\\
             0&0&0&1&0\\
             0&0&0&0&1\\
             0&1&0&0&-1\\
             0&0&1&0&-1\\
             0&0&0&1&-1\\
             0&-1&-1&0&1\\
             0&-1&0&-1&1\\
             0&0&-1&-1&1
        \end{pmatrix}\in\mathbb{R}^{10\times 5}\quad\text{and}\quad
        b^\top = \begin{pmatrix}
            0\\
            1\\
            1\\
            1\\
            2\\
            0\\
            0\\
            0\\
            0\\
            0\\
            0
        \end{pmatrix}\in\mathbb{R}^{10}.
    \]

    By making use of the computer algebra system \textsc{OSCAR} \cite{OSCAR}, we can compute all the lattice points of $\mathcal{P}_2^2$ and obtain the following six:
    \[
        (0,0,0,0,0), (0,1,1,1,1), (0,1,1,1,2), (0,1,1,0,1), (0,1,0,1,1), (0,0,1,1,1). 
    \]

    Each of these points corresponds to one of the six $q$-matroids on $\F_2^2$. Specifically, the first three correspond to the uniform $q$-matroids $\mU_{0,2}(2)$, $\mU_{1,2}(2)$, and $\mU_{2,2}(2)$, respectively. The last three correspond to the $q$-matroids of rank $1$ possessing exactly two bases. Furthermore, we find that these six lattice points constitute the entire vertex set of $\mP_2^2$.
\end{example}

\begin{remark} As \Cref{ex:first_ex} illustrates, some inequalities describing $\mathcal{P}_q^n$ are redundant. In particular, we can make the following observations. 
\begin{enumerate}
    \item[(a)] For inequalities of type $2$, it is sufficient to consider pairs of subspaces $X,Y\in\mL(E)^\ast$ such that $X \lessdot Y$ in the lattice $\mL(E)$, i.e. $Y$ covers $X$. This is because for $X\lessdot Z\lessdot Y$ we have inequalities $v_X\leq v_Z$ and $v_Z\leq v_Y$ that already imply $v_X\leq v_Y$.
    \item[(b)] For inequalities of type $3$, we do not need to consider pairs $X,Y\in\mL(E)$ such that $X\leq Y$ (or $Y\leq X$). This is because we would get $v_{X+Y}+v_{X\cap Y}=v_{X}+v_{Y}$, since $X+Y = Y$ and $X\cap Y=X$ (or $X+Y = X$ and $X\cap Y=Y$).  
\end{enumerate}
\end{remark}

Another important observation is that we always have the equality $v_{\la0\ra}=0$ by the first rank axiom. 
Consequently, the $q$-rank polytope is contained in a hyperplane of $\mathbb{R}^T$ and can be naturally embedded in $\mathbb{R}^{T-1}$. We denote the polytope obtained by removing the coordinate corresponding to $\langle 0 \rangle$ as the \textbf{reduced $q$-rank polytope}, $\rRP$. Note that $\rRP$ preserves all algebraic and geometric properties of $\RP$. Throughout this paper, we will explicitly indicate whether we are working with $\RP$ or $\rRP$. We now proceed to determine the geometric properties of $\rRP$, specifically focusing on its dimension and its set of lattice points.

Before determining the dimension of the reduced $q$-rank polytope, we recall that the dimension of a polytope $\mathcal{P}$, $\dim(\mathcal{P})$, is defined as the dimension of its \textbf{affine hull}, which is the smallest affine subspace containing $\mathcal{P}$.
If $\mathcal{P} \subseteq \mathbb{R}^m$ is defined by a system of linear inequalities, its dimension is $m - d$, where $d$ is the number of linearly independent \textbf{implicit equalities} (i.e., defining inequalities that are satisfied as equalities by all points in $\mathcal{P}$). A crucial property is that $\mathcal{P}$ has a non-empty interior relative to its affine hull; furthermore, the existence of a point satisfying all defining inequalities strictly (except for the implicit equalities) implies that no other implicit equalities exist. For further details, we refer to \cite{ziegler2012lectures, schrijver1998theory}.

\begin{proposition}\label{prop:dimension}
The reduced $q$-rank polytope $\rRP$ has dimension $T-1$.
\end{proposition}
\begin{proof}
By definition, $\rRP \subset \mathbb{R}^{T-1}$, which immediately implies that $\dim(\rRP) \le T-1$. To prove the equality, it suffices to show that $\rRP$ has a non-empty interior in $\mathbb{R}^{T-1}$. This is guaranteed if there exists a point $P_\rho \in \rRP$ that satisfies all defining linear inequalities of the reduced $q$-rank polytope strictly.
Let $\rho$ be a real-valued function on $\mL(E)$ defined by $\rho(X)=\frac{\dim(X)}{\dim(X)+1}$ for every $X\in\mL(E)$. Then $\rho$ is a $q$-rank function. In fact, (R1) and (R2) are clearly satisfied with strict inequalities. In order to show (R3), consider two spaces $X,Y\in\mL(E)$, such that $X\nleq Y$ or $Y\nleq X$. For ease of notation, let $x=\dim(X)$, $y=\dim(Y)$, $z=\dim(X\cap Y)$ and $t=\dim(X+Y) = x+y-z$. In particular, $0< z< \min\{x,y\}$ and $x,y<t$. We show that 
$\frac{t}{t+1} + \frac{z}{z+1} < \frac{x}{x+1} + \frac{y}{y+1}.$
We then have the following 
$$\begin{array}{cccc}
    & \frac{t}{t+1} + \frac{z}{z+1} & < & \frac{x}{x+1} + \frac{y}{y+1} \\
    \Leftrightarrow & \frac{1}{t+1} + \frac{1}{z+1} & > &\frac{1}{x+1} + \frac{1}{y+1}\\
    \Leftrightarrow & \frac{x+y+2}{(x+y-z+1)(z+1)} &> & \frac{x+y+2}{(x+1)(y+1)}\\
    \Leftrightarrow & (x+y-z+1)(z+1) &< &(x+1)(y+1)\\
    \Leftrightarrow & (x-z)(z-y) &<& 0.\\
\end{array}$$
The last inequality is always satisfied, so also (R3) is satisfied with strict inequality. 
Because the point $P_\rho$ satisfies all defining inequalities strictly, it lies in the interior of $\rRP$ relative to $\mathbb{R}^{T-1}$. Consequently, the affine hull of $\rRP$ is $\mathbb{R}^{T-1}$, and $\dim(\rRP) = T-1$.
\end{proof}

Determining in general the faces of $\rRP$ is not an easy task. In the following example, we examine the case of $\bar{\mathcal{P}}_2^2$, which is the smallest one we can deal with.

\begin{example}
Consider the $q$-rank polytope $\mP_2^2$ in \Cref{ex:first_ex}. First of all, notice that $\bar{\mathcal{P}}_2^2\subset\mathbb{R}^4$ and, by using the lexicographic linear order, we have $\la 01\ra\preceq\la 10\ra\preceq\la 11\ra\prec E$. As already established in \Cref{prop:dimension}, the dimension of $\bar{\mathcal{P}}_2^2$ is $4$.     A face of a polytope is any intersection of the polytope with a supporting hyperplane. By using the computer algebra system \textsc{OSCAR} \cite{OSCAR} we compute that there exist $6$ many $0$-dimensional faces, $15$ many $1$-dimensional faces, $19$ many $2$-dimensional faces and $9$ many $3$-dimensional faces of $\bar{\mathcal{P}}_2^2$. We illustrate some of them.

\noindent \underline{\textbf{$0$-faces.}} For all $v\in\bar{\mathcal{P}}_2^2$ we have that $v_1+v_2+v_3+v_4\leq 5$. Then 
        \[
            F_0 = \bar{\mathcal{P}}_2^2\cap\{v\in\mathbb{R}_{\geq 0}^4\;\mid\;v_1+v_2+v_3+v_4=5\}=\{(1,1,1,2)\}
        \]
        is $0$-dimensional and yields a vertex of $\bar{\mathcal{P}}_2^2$. 
        Moreover, if we consider the intersections of  $\bar{\mathcal{P}}_2^2$ with the hyperplanes $\{v\in\mathbb{R}_{\geq 0}^4\;\mid\;v_2+v_3+v_4=4\}$ and $\{v\in\mathbb{R}_{\geq 0}^4\;\mid\;v_1+v_3+v_4=4\}$
        we get different descriptions of $F_0$.
        We also have that $v_i+v_4\leq 3$, for $i=1,2,3$. 
        The intersections of $\bar{\mathcal{P}}_2^2$ with the hyperplanes $H_i=\{v\in\mathbb{R}_{\geq 0}^4\;\mid\;v_i+v_4=3\}$, give $F_0$, for each $i=1,2,3$.
        Finally, we have that $v_4\leq 2$. If we consider the intersection of $\bar{\mathcal{P}}_2^2$ with the hyperplane $\{v\in\mathbb{R}_{\geq 0}^4\;\mid\;v_4=2\}$, we get again the same face $F_0$.
        
\noindent \underline{\textbf{$1$-faces.}} For all $v\in\bar{\mathcal{P}}_2^2$ we have that $v_1+v_2+v_3\leq 3$. Then
    \begin{align*}
        F_1 &= \bar{\mathcal{P}}_2^2\cap\{v\in\mathbb{R}_{\geq 0}^4\;\mid\;v_1+v_2+v_3=3\}\\
        &=\{(1,1,1,v_4)\;\mid\; 1\leq v_4\leq 2\}
    \end{align*}
is $1$-dimensional and contains only two $q$-matroids, namely $\mU_{2,2}(2)$ and $\mU_{1,2}(2)$. 

    \noindent \underline{\textbf{$2$-faces.}} For all $v\in\bar{\mathcal{P}}_2^2$ we have that $v_1+v_2\leq 2$.
        Then
        \begin{align*}
            F_{2,1} &= \bar{\mathcal{P}}_2^2\cap\{v\in\mathbb{R}_{\geq 0}^4\;\mid\;v_1+v_2=2\}\\
            &=\{(1,1,v_3,v_4)\;\mid\; 0\leq v_3\leq 1,\;1\leq v_4\leq 2,\;v_4-v_3\leq 1\}
        \end{align*}
        is $2$-dimensional and contains the lattice points $(1,1,1,2),(1,1,1,1),(1,1,0,1)$, which correspond to $\mU_{2,2}(2)$, $\mU_{1,2}(2)$ and the $q$-matroid, having loop space $\la 11 \ra$, respectively.
        Similarly, we can consider 
        \begin{align*}
            F_{2,2} &= \bar{\mathcal{P}}_2^2\cap\{v\in\mathbb{R}_{\geq 0}^4\;\mid\;v_1+v_3=2\}\\
            &=\{(1,v_2,1,v_4)\;\mid\; 0\leq v_2\leq 1,\;1\leq v_4\leq 2,\;v_4-v_2\leq 1\}
        \end{align*}
        and 
        \begin{align*}
            F_{2,3} &= \bar{\mathcal{P}}_2^2\cap\{v\in\mathbb{R}_{\geq 0}^4\;\mid\;v_2+v_3=2\}\\
            &=\{(v_1,1,1,v_4)\;\mid \; 0\leq v_1\leq 1,\;1\leq v_4\leq 2,\;v_4-v_1\leq 1\}.
        \end{align*}
        We get that $F_{2,2}$ and $F_{2,3}$ are also $2$-faces containing both $\mU_{1,2}(2)$ and $\mU_{2,2}(2)$. 
        $F_{2,2}$ also contains the $q$-matroid with loop space $\la 01 \ra$ and $F_{2,3}$ contains the $q$-matroid with loop space $\la 10 \ra$ respectively.

    \noindent \underline{\textbf{$3$-faces.}} For all $v\in\bar{\mathcal{P}}_2^2$ we have that $v_i\leq 1$, for $i=1,2,3$.
        Then for each $i=1,2,3$ we have that
        \[
            F_{3,i} = \bar{\mathcal{P}}_2^2\cap\{v\in\mathbb{R}_{\geq 0}^4\;\mid\;v_i=1\}
        \]
        is $3$-dimensional. $F_{3,1}$ contains the four lattice points $(1,1,1,2)$, $(1,1,1,1)$, $(1,1,0,1)$, $(1,0,1,1)$. $F_{3,2}$ contains the lattice points $(1,1,1,2), (0,1,1,1), (1,1,0,1), (1,1,1,1)$ and $F_{3,3}$ contains the lattice points $(1,1,1,2), (0,1,1,1), (1,0,1,1), (1,1,1,1)$. 
        
    Finally recall that $\bar{\mathcal{P}}_2^2$ has only six lattice points, which correspond to all the different $q$-matroids on $\F_2^2$ and these points are the vertices of $\bar{\mathcal{P}}_2^2$. In particular, this means that there are no interior lattice points.  
\end{example}

\begin{example}
    In this example, we collect some computational results obtained with \textsc{OSCAR} for $\bar{\mathcal{P}}_3^2$ and $\bar{\mathcal{P}}_2^3$. 

    \begin{center}
        \begin{tabular}{c|c|c}
             & $\bar{\mathcal{P}}_3^2$ & $\bar{\mathcal{P}}_2^3$\\ \hline
             \textbf{Dimension} & $5$  & $15$  \\
             \textbf{Number of lattice points} & $7$  & $32$ \\
             \textbf{Number of interior lattice points} & $0$ & $0$  \\
             \textbf{Number of integer vertices} & 7 & 32\\ 
              \textbf{Number of vertices} & $11$ & $3483$ \\
        \end{tabular}
    \end{center}
\end{example}

We note that also for $\bar{\mathcal{P}}_3^2$ and $\bar{\mathcal{P}}_2^3$ we have that all nonzero lattice points are vertices. However,  we have, in addition, other rational vertices. This is not just a coincidence, but a property that holds in general. 
The next result is the main theorem of this section, in which we show that $\rRP$ has no interior lattice points. 
We denote by $\text{Vert}(\rRP)$ the set of vertices of $\rRP$.

\begin{theorem}\label{thm:integer_vertices}
    Let $\mM=(\mL(E),\rho)$ be a $q$-matroid and let $p_\mM\in\rRP$ be its corresponding point. Then $p_\mM\in\text{Vert}(\rRP)$. 
\end{theorem}

\begin{proof}
    Let $\mI$ and $\mD$ denote the collections of independent and dependent spaces of $\mM$ respectively. Let $L$ be the loop space of $\mM$. By \Cref{prop:dimension} we need to find $T-1$ independent supporting hyperplanes $H_1,\ldots,H_{T-1}$ such that $\rRP\subset H_i^+$ or $\rRP\subset H_i^-$ for all $i=1,\ldots,T-1$, and such that their intersection is given by the point $p_\mM$.

    So we choose the following hyperplanes:
    \begin{enumerate}
        \item For all $X\in\mI\setminus\{\la0\ra\}$ we consider the hyperplane given by the equation $v_X=\dim(X)$, having normal vector $e_X\in\mathbb{R}^{T-1}$.

        \item For all $X\in\mD\setminus\mL(L)$, let $Y$ be a space in $\mI$ such that~$Y<X$ and $Y$ is inclusion-maximal in $X$ with this property. We consider then the hyperplane given by $v_Y-v_X=0$, having normal vector $e_Y-e_X\in\mathbb{R}^{T-1}$.

        \item Finally for all $X\in\mL(L)$ we consider the hyperplane given by $v_X=0$, having normal vector $e_X\in\mathbb{R}^{T-1}$.
    \end{enumerate}

    Now since $\mI\setminus\{\la0\ra\}\cup\mD\setminus\mL(L)\cup\mL(L)=\mL(E)^\ast$ we have chosen $T-1$ different hyperplanes. Next, we need to show that their normal vectors are linearly independent. Therefore, we write them as rows of a matrix $A$ and, since every $X\in\mL(E)^\ast$ has a corresponding hyperplane, we can order them, from top to bottom, in the same way we ordered our coordinates (i.e. columns of $A$). Then $A$ is a lower triangular matrix with $1$ and $-1$ on the diagonal. Thus $\det(A)\ne 0$ and so $A$ has rank $T-1$. This implies that the linear system given by $Av=b$ has a unique solution, where the coordinate order of $b$ matches the one of the normal vectors and $b$ has entry $\dim(X)$ if we are in the case (1) and entry $0$ anywhere else.

    At this point the only thing left to show is that the solution to $Av=b$ is given by our point $p_\mM$. Using Gaussian elimination for the extended coefficient matrix $(A,b)$ we can eliminate the entries equal to $1$ in rows coming from  case (2) in the following way.

    Let $R$ be such a row, having a $1$ and a $-1$. Then the $1$ corresponds to a space $Y\in\mI\setminus\{\la0\ra\}$ such that $Y<X$ is inclusion-maximal in $X$, where $X$ corresponds to the $-1$. Since $Y\in\mI\setminus\{\la0\ra\}$, there exists a row $\Tilde{R}$ coming from the normal vector of a case (1) hyperplane associated to $Y$. Moreover such a row $\Tilde{R}$ has $b$ entry equal to $\dim(Y)$. So $R-\Tilde{R}$ yields the equation
    \[
        -v_X = -\dim(Y) \Longleftrightarrow v_X = \dim(Y).
    \]
    This means that $v_X = \rho(X)$ by definition of rank in terms of independent spaces. Now since we chose the hyperplanes of case (2) all in this way, we get $v_X =\rho(X)$ for all $X\in\mD\setminus\mL(L)$. Putting it all together, we get that the solution of $Av=b$ is given by the point $p=(v_X)_{X\in\mL(E)^\ast}$ such that
    \[
        v_X =
        \begin{cases}
        \;\dim(X)& \text{if } X\in\mI\setminus\{\la0\ra\},\\
        \;0& \text{if } X\in\mL(L),\\
        \;\max(\dim(Y)\;\mid\;Y<X,Y\in\mI\setminus\{\la0\ra\})& \text{if } X\in\mD\setminus\mL(L).
        \end{cases}
    \]
    Thus
    \[
        p=(\rho(X))_{X\in\mL(E)^\ast}=p_\mM
    \]
    and therefore $p_\mM\in\text{Vert}(\rRP)$, which completes our proof.
\end{proof}

In \Cref{thm:integer_vertices}, we showed that $\rRP$ has no integer interior points. In particular, the points corresponding to $q$-matroids are necessarily vertices in the reduced $q$-rank polytope. We leave as an open question to characterize the rational vertices of $\rRP$.

\begin{remark}
    We want to emphasize once more that the $q$-rank polytope $\RP$ has the same algebraic and geometric properties as the reduced $q$-rank polytope $\rRP$. This especially means that \Cref{prop:dimension} and \Cref{thm:integer_vertices} hold for $\RP$ as well. So, in the remainder of the paper, we apply them also to $\RP$.
\end{remark}

\section{The convex combination of two \emph{q}-(poly)matroids}\label{sec: conv_combi}

In this section, we define a new operation on $q$-polymatroids, by looking at their representation as points in the $q$-rank polytope. We initiate the study of the properties of this operation.

The $q$-rank polytope $\RP$ is, by construction, a convex set. This implies that for any collection of points $p_1, \ldots, p_t \in \RP$, their convex combination $p = \sum\limits_{i=1}^t \lambda_i p_i$, where the coefficients $\lambda_i \in\R_{> 0}$ and $\sum\limits_{i=1}^t  \lambda_i = 1$, is also contained in $\RP$. Given the linear ordering $\preceq$ on $\mL(E)$ as defined in Section~\ref{sec: polytope}, every point in $\RP$ corresponds to a unique $q$-rank function. Since $\RP$ is a convex set, any such combination consistently yields a vector whose coordinates satisfy the $q$-rank axioms, thereby defining another $q$-polymatroid.

\begin{definition}\label{def: conv_combi}
    Let $\mM_1, \ldots, \mM_t$ be $t$ different $q$-polymatroids with ground space $E$. Let $p_{\mM_1}, \ldots, p_{\mM_t}\in \RP$  be their corresponding points in the $q$-rank polytope. For every $\lambda_1,\ldots,\lambda_t\in\R_{> 0}$ such that $\sum\limits_{i=1}^t\lambda_i=1$, the \textbf{convex combination with coefficients} $\lambda_1,\ldots, \lambda_t$ of $\mM_1, \ldots, \mM_t$ is the $q$-polymatroid $\lambda_1\mM_1 + \cdots +\lambda_t\mM_t$, corresponding to the convex combination $\lambda_1 p_{\mM_1}+\cdots+\lambda_tp_{\mM_t}\in\RP$.
    \end{definition}
    In the specific case of two $q$-polymatroids $\mathcal{M}_1$ and $\mathcal{M}_2$, depending on the context, we may also fix $\lambda := \lambda_1$ and write $\lambda\mathcal{M}_1 + (1-\lambda)\mathcal{M}_2$; we refer to this as the \textbf{convex combination with coefficient $\lambda$}.

\begin{remark}
    Let $\rho_1,\ldots,\rho_t$ be the $q$-rank functions of $\mM_1,\ldots\mM_t$. The $q$-rank function of $\sum\limits_{i=1}^t\lambda_i\mM_i$ is given by $\sum\limits_{i=1}^t\lambda_i\rho_i(X)$, for every $X\in\mL(E)$. The fact that this is indeed a $q$-rank function is immediate by the definition of $\RP$. 
\end{remark}

We recall the notion of \emph{duality} for $q$-polymatroids. 
\begin{definition}
    Let $\mathcal{M} = (\mL(E), \rho)$ be a $q$-polymatroid and set 
    \begin{equation}
        \rho^*(V) = \dim(V) + \rho(V^{\perp}) - \rho(E).
    \end{equation}
    Then $\rho^*$ is a $q$-rank function on $E$ and $\mM^* = (\mL(E), \rho^*)$ is a $q$-polymatroid. It is called the \textbf{dual} of $\mM$.
\end{definition}
We have the following result. 
\begin{proposition}\label{prop:duality}
    Let $\mM_1 = (\mL(E), \rho_1)$ and $\mM_2 = (\mL(E), \rho_2)$ be $q$-polymatroids on $E$. Let $\lambda_1, \lambda_2 \in \mathbb{R}_{>0}$ be such that $\lambda_1 + \lambda_2 = 1$. Let $\mM = (\mL(E), \rho)$ be the convex combination of $
    \mM_1$ and $\mM_2$, with coefficients $\lambda_1,\lambda_2$ and $\rho = \lambda_1 \rho_1 + \lambda_2 \rho_2$. Then
    $$ \mM^* = \lambda_1 \mM_1^* + \lambda_2 \mM_2^*.$$
\end{proposition}
\begin{proof}
Let $\rho^*$ denote the $q$-rank function of $\mathcal{M}^*$. By the definition of the dual $q$-polymatroid, for any $V \in \mathcal{L}(E)$, we have
\begin{align*}
    \rho^*(V) &= \dim(V) +\rho(V^\perp)-\rho(E)\\
    &=\dim(V) +(\lambda_1\rho_1 + \lambda_2\rho_2)(V^\perp)-(\lambda_1\rho_1 + \lambda_2\rho_2)(E).
\end{align*}
Using the fact that $\lambda_1 + \lambda_2 = 1$, we can write $\dim(V) = (\lambda_1 + \lambda_2)\dim(V)$. Hence, substituting this in the equation above, we have
\begin{align*}
    \rho^*(V) &=(\lambda_1+\lambda_2)\dim(V) +(\lambda_1\rho_1 + \lambda_2\rho_2)(V^\perp)-(\lambda_1\rho_1 + \lambda_2\rho_2)(E)\\
    &=\lambda_1(\dim(V) + \rho_1(V^\perp) -\rho_1(E)) + \lambda_2(\dim(V) + \rho_2(V^\perp) -\rho_2(E)) \\
    &= \lambda_1\rho_1^*(V) + \lambda_2\rho_2^*(V).
\end{align*}
From this, we conclude that $\mM^* = \lambda_1\mM_1^* + \lambda_2\mM_2^*$.
\end{proof}

In the following, we study some structural properties of the convex combination of two $q$-polymatroids $\mM_1,\mM_2$.
In particular, we investigate the structure of the flats. By restricting to the convex combination of two $q$-matroids, we can also analyze cyclic spaces and $\mu$-independence (in the latter case the coefficients of the convex combination have to be rational).

\subsection{The flats}\label{subsec: flats}

In this subsection, we show that the flats of two $q$-polymatroids are flats of any convex combination, independently from the choice of coefficients.

Recall that we define a closure operator for $q$-polymatroids as a direct extension of the one for $q$-matroids, as follows.

    \begin{definition}\label{def-closure}
        Let $(\mL(E),\rho)$ be a $q$-polymatroid.
        For each $A \in \mL(E)$, define $$\Cl_\rho(A):=\{x \in \mL(E)\mid \rho(A+x)=\rho(A)\}.$$
    	The \textbf{closure function} of $(\mL(E),\rho)$ is the function defined by
    	\[\cl_\rho:\mathcal{L}(E) \to\mathcal{L}(E), \  A \mapsto \cl_\rho(A)=\sum_{x \in \Cl_\rho(A)} x.
    	\]
    \end{definition}
Let $X\in\mL(E)$. It is easy to see that the closure of $X$ is a flat of the $q$-polymatroid $(\mL(E),\rho)$, and that if $X$ is a flat, then $X = \cl(X)$. Clearly, $\Cl_\rho(F)\supseteq \mL(F)_1$.

\begin{proposition}\label{prop:flats_conv_comb}
Let $\mM_1=(\mL(E),\rho_1), \; \mM_2=(\mL(E),\rho_2)$ be $q$-polymatroids. 
Then, $F\in\mL(E)$ is a flat of any convex combination of $\mM_1$ and $\mM_2$  if and only if $\Cl_{\rho_1}(F)\cap\Cl_{\rho_2}(F)=\mL(F)_1$.
\end{proposition}
\begin{proof}
 Let $\lambda_1,\lambda_2\in\R_{>0}$, be such that $\lambda_1+\lambda_2=1$ and  $\rho$ be the $q$-rank function of $\lambda_1\mM_1+\lambda_2\mM_2$.
 
 ({$\Longleftarrow$}) Assume that $\Cl_{\rho_1}(F)\cap\Cl_{\rho_2}(F)=\mL(F)_1$. We will show that $F$ is a flat of $\lambda_1\mM_1+\lambda_2\mM_2$, i.e.  for every $x\in\mL(E)_1\setminus\mL(F)_1$, we have $\rho(F+x)>\rho(F)$. Let $x\in\mL(E)_1\setminus\mL(F)_1$. By our hypothesis, $x$ cannot belong to both $\Cl_{\rho_1}(F)$ and $\Cl_{\rho_2}(F)$. If $x\not\in\Cl_{\rho_1}(F)$, then 
\begin{align*}
    \rho(F+x) &= \lambda_1\rho_1(F+x) + \lambda_2\rho_2(F+x) \\
    &> \lambda_1\rho_1(F) + \lambda_2\rho_2(F+x) \\
    &\geq \lambda_1\rho_1(F) + \lambda_2\rho_2(F)=\rho(F).
\end{align*}
A similar reasoning can be done for every $x\not\in\Cl_{\rho_2}(F)$. Hence, we conclude that $F$ is a flat of $\lambda_1\mM_1 +\lambda_2\mM_2$.

({$\Longrightarrow$}) Assume by contradiction that $\Cl_{\rho_1}(F)\cap\Cl_{\rho_2}(F)\ne\mL(F)_1$. 
Then there exists some $x \in \mL(E)_1 \setminus \mL(F)_1$ such that $x \in \Cl_{\rho_1}(F)$ and $x \in \Cl_{\rho_2}(F)$. This implies $\rho_1(F+x) = \rho_1(F)$ and $\rho_2(F+x) = \rho_2(F)$. Consequently, 
$$\rho(F+x) = \lambda_1\rho_1(F+x) + \lambda_2\rho_2(F+x) =\lambda_1\rho_1(F) + \lambda_2\rho_2(F) = \rho(F),$$
from which we conclude that $F$ is not a flat of $\lambda_1\mM_1+\lambda_2\mM_2$ (see \Cref{def:flats}). 
\end{proof}

As a result, we obtain the following immediate corollary.

\begin{corollary}\label{coro:flats_conv_comb}
   Let $\mM_1=(\mL(E),\rho_1), \; \mM_2=(\mL(E),\rho_2)$ be $q$-polymatroids.  Let $\mF_i$ denote the collection of flats of $\mM_i$, for $i=1,2$. Then every $F$ in $\mF_1\cup\mF_2$ is a flat of any convex combination of $\mM_1$ and $\mM_2$.
\end{corollary}
\begin{proof}
    If $F\in\mF_1\cup\mF_2$ then  $\Cl_{\rho_1}(F)\cap\Cl_{\rho_2}(F)=\mL(F)_1$. The rest follows from \Cref{prop:flats_conv_comb}. 
\end{proof}

\subsection{The cyclic spaces}\label{subsec: cyclic}

The study of cyclic spaces of the convex combination of two $q$-polymatroids is not as immediate. Hence, in this subsection, we focus on the cyclic spaces of the convex combination of two $q$-matroids. 

We introduce the following notation.  Let $\mM_1=(\mL(E),\rho_1), \mM_2=(\mL(E),\rho_2)$ be two $q$-matroids. Let $\mO_i$ denote the collection of cyclic spaces of $\mM_i$, for $i=1,2$. Further, for a space $X\in\mL(E)$ we define $\mN_i(X) =\{H\in\mH(X) : \rho_i(H)<\rho_i(X)\}$, for $i=1,2$.

\begin{proposition}\label{prop:cyc_spaces_noloops}
    Let $\mM_1=(\mL(E),\rho_1), \mM_2=(\mL(E),\rho_2)$ be two $q$-matroids {without loops}. $X\in\mL(E)$ is a cyclic space of any convex combination of $\mM_1$ and $\mM_2$ if and only if $\mN_1(X)\cap\mN_2(X)=\emptyset$.
\end{proposition}
\begin{proof}
     Let $\lambda_1,\lambda_2\in\R_{>0}$ be such that $\lambda_1+\lambda_2=1$ and $\rho$ be the $q$-rank function of $\lambda_1\mM_1+\lambda_2\mM_2$.
    
    $(\Longleftarrow)$  
    Let $X\in\mL(E)$, be such that $\mN_1(X)\cap\mN_2(X)=\emptyset$.
    We distinguish different cases. 
    
    If $X\in\mO_1\cap\mO_2$, for every $H\in\mH(X)$ we have $\rho_i(H)=\rho_i(X)$ for $i=1,2$, since $\mM_1$ and $\mM_2$ are $q$-matroids, and we have
    $$\rho(H) = \lambda_1\rho_1(H) + \lambda_2\rho_2(H) = \lambda_1\rho_1(X) + \lambda_2\rho_2(X) = \rho(X).$$
    Hence, $X$ is a cyclic space for $\lambda_1\mM_1+\lambda_2\mM_2$, according to \Cref{def:cyc_spaces}, part (1).
    
    Now assume that $X\in\mO_1$ and $X\notin \mO_2$ (an analogous reasoning can be done for $X\in\mO_2$ and $X\notin \mO_1$). Then, for every $H\in\mH(X)$ we have $\rho_1(H)=\rho_1(X)$ and there exists $H'\in\mH(X)$ such that $\rho_2(H^\prime)<\rho_2(X)$. More specifically, since $\mM_2$ is a $q$-matroid, for every $a\in\mL(X)_1\setminus\mL(H')_1$ we have $\rho_2(X)=\rho_2(H'+a) = \rho_2(H')+1$. Clearly, part (1) of \Cref{def:cyc_spaces} cannot hold for such~$H'$. We need to verify that for some $a\in\mL(X)_1\setminus\mL(H')_1$, we have that $0<\rho(X)-\rho(H')<\rho(a)$. The first inequality is immediate. For the second one we have
    \begin{align*}
        \rho(X)-\rho(H') &= \lambda_1\rho_1(X) + \lambda_2\rho_2(X) -\lambda_1\rho_1(H') - \lambda_2\rho_2(H') \\
        &=\lambda_1(\rho_1(X) - \rho_1(H')) + \lambda_2(\rho_2(X) - \rho_2(H')) \\
        &=\lambda_2 < 1 = \lambda_1+\lambda_2 = \lambda_1\rho_1(a) + \lambda_2\rho_2(a) =\rho(a),
    \end{align*}
    for every $a\in\mL(X)_1\setminus\mL(H')_1$.
    
    Finally, assume that $X\notin\mO_1\cup\mO_2$, but $\mN_1(X)\cap\mN_2(X)=\emptyset$. Then,
    for every $H\in\mH(X)$, we have that either $\rho_1(H)<\rho_1(X)$ or $\rho_2(H)<\rho_2(X)$, but not both. In particular, for each $H\in\mH(X)$ we can argue as in the previous case.

   $(\Longrightarrow)$ Assume by contradiction that there exists some $H'\in\mN_1(X)\cap\mN_2(X)$. Then $X$ cannot be a cyclic space of $\lambda_1\mM_1+\lambda_2\mM_2$, since for all  $a\in\mL(X)_1\setminus\mL(H')_1$ we would have
    \begin{align*}
        \rho(X)-\rho(H') &= \lambda_1\rho_1(X) + \lambda_2\rho_2(X) -\lambda_1\rho_1(H') - \lambda_2\rho_2(H') \\
        &=\lambda_1(\rho_1(X) - \rho_1(H')) + \lambda_2(\rho_2(X) - \rho_2(H')) \\
        &=\lambda_1+\lambda_2 =\lambda_1\rho_1(a) +\lambda_2\rho_2(a) =\rho(a).
        \end{align*}
\end{proof}

As an immediate result we obtain the following corollary. 

\begin{corollary}\label{coro:cyc_spaces_noloops}
     Let $\mM_1, \ \mM_2$ be two $q$-matroids without loops. 
     Then every $X\in\mO_1\cup\mO_2$ is a cyclic space of any convex combination of $\mM_1$ and $\mM_2$.
\end{corollary}

\begin{remark}
    If $\mM_1=(\mL(E),\rho_1)$ or $\mM_2=(\mL(E),\rho_2)$ have loops we cannot guarantee that all the spaces in $\mO_1\cup\mO_2$ are cyclic in a convex combination of $\mM_1$ and $\mM_2$. Let $\lambda_1,\lambda_2\in\R_{>0}$, be such that $\lambda_1+\lambda_2=1$ and $\rho$ be the $q$-rank function of $\lambda_1\mM_1+\lambda_2\mM_2$.
    Assume that $\mM_2$ has no loops, and $X$ is the loop space of $\mM_1$. Then $X\in\mO_1$, and assume that $X\notin\mO_2$. Let $H\in\mH(X)$ be such that $\rho_2(H)=\rho_2(X)-1$. We have that
    \begin{align*}
        \rho(X)-\rho(H) &= \lambda_1\rho_1(X) + \lambda_2\rho_2(X) -\lambda_1\rho_1(H) - \lambda_2\rho_2(H) \\
        &=\lambda_1(\rho_1(X) - \rho_1(H)) + \lambda_2(\rho_2(X) - \rho_2(H)) \\
        &=\lambda_2 = \lambda_1\rho_1(a) + \lambda_2\rho_2(a) =\rho(a),
    \end{align*}
    for every $a\in\mL(X)_1\setminus\mL(H)_1$. The last equality is because $\rho_1(a)=0$ and $\rho_2(a)=1$. In this case,~$X$ cannot be a cyclic space of $\lambda_1\mM_1+\lambda_2\mM_2$.
\end{remark}

Corollaries \ref{coro:flats_conv_comb} and \ref{coro:cyc_spaces_noloops} immediately yield us the following property of cyclic flats of a convex combination of the $q$-matroids.

\begin{corollary}
    Let $\mM_1, \mM_2$ be two $q$-matroids without loops. Let $\mZ_i$ denote the collection of cyclic flats of $\mM_i$, for $i=1,2$. Then every $X\in\mZ_1\cup\mZ_2$ is a cyclic flat of any convex combination of $\mM_1$ and $\mM_2$.
\end{corollary}

By following \cite[Def. 7.7]{gluesing2024decompositions}, we introduce the notion of \emph{full} $q$-polymatroid.

\begin{definition}
    Let $\mM$ be a $q$-polymatroid on $E$. We say that $\mM$ is \textbf{full} if $\langle 0\rangle$ is a flat and~$\F_q^n$ is a cyclic space of $\mM$.
\end{definition}

The following result is another direct consequence of \Cref{coro:flats_conv_comb} and \Cref{coro:cyc_spaces_noloops}.

\begin{proposition}
    Let $\mM_1,\mM_2$ be two $q$-matroids with ground space $E$. Assume that $\mM_1$ is full and $\mM_2$ has no loops. Then any convex combination $\mM$ of $\mM_1$ and $\mM_2$ is full. 
\end{proposition}
\begin{proof}
    We know that the zero space $\langle 0\rangle$ is a flat of both $\mM_1$ and $\mM_2$, so by \Cref{coro:flats_conv_comb} it is a flat of $\mM$. Moreover $\mM_1$ is full, which means in particular that it has no loops and $E$ is a cyclic space of this $q$-matroid. By \Cref{coro:cyc_spaces_noloops} we then get that $E$ is cyclic in $\mM$.
\end{proof}

\subsection{The $\mu$-independent spaces}\label{subsec: mu_indeps}

In this subsection,  we turn to the investigation of the $\mu$-independent spaces of rational $q$-polymatroids arising as convex combination of two $q$-matroids. For this subsection, we let $\mM_1=(\mL(E),\rho_1), \mM_2=(\mL(E),\rho_2)$ be two $q$-matroids. Let $\mI_i$ denote the collection of independent spaces of $\mM_i$, for $i=1,2$. Moreover, we set $\mM$ to be the convex combination of $\mM_1$ and $\mM_2$ with rational coefficient $\lambda = \frac{a}{b}$, where $a\geq 1 $ and $b>a$, as in \Cref{def: conv_combi}. Then $\mM$ admits a denominator $\mu$, since $\mM_1$ and $\mM_2$ correspond to integer points in $\RP$.

By definition of $\mM$ we have that $b$ is always a denominator for $\mM$, but in general it might not be principal. 
However, we have the following result that provides sufficient conditions for $\mu$ to be principal.

\begin{theorem}
    Assume that there exists a one-dimensional subspace $x\in\mL(E)$ such that $x$ is a strong independent space in $\mM$. Then $b$ is a principal denominator of $\mM$.
\end{theorem}
\begin{proof}
By \cite[Rem. 2.7]{gluesing2022q}, if $x\in\mL(E)$ is a strong independent space in $\mM$, then any denominator of $\mM$ is integer. By the definition of $\mM$ we have that $b$ is always a denominator for $\mM$. Since the independent spaces determine a $q$-matroid (see e.g. \cite{byrne2022constructions}), we see that there exists a space $X\in\mL(E)$, such that $X$ is, without loss of generality, independent in $\mM_1$, but not in $\mM_2$. Then, there is a space $C\leq X$, which is a circuit in $\mM_2$, but $C$ is independent in $\mM_1$. Then $\rho_2(C)=\dim(C)-1$ and $\rho_1(C) = \dim(C)$. 
Let $\rho$ be the rank function of $\frac{a}{b}\mM_1 + (1-\frac{a}{b})\mM_2$.
Then,
$$\rho(C) = \frac{a}{b}\rho_1(C) + \left(1-\frac{a}{b}\right)\rho_2(C) = \rho_2(C) + \frac{a}{b}\left(\rho_1(C) - \rho_2(C)\right) = \rho_2(C) + \frac{a}{b}.$$
Then, $b$ is the smallest integer such that $\mu\rho(C)$ is an integer by the generalized Euclidean lemma. Hence, we conclude that $\mu$ is the principal denominator of $\mM$.
\end{proof}

The following proposition describes some of the $b$-independent spaces of $\mM$, and it will be useful later on.

\begin{proposition}\label{prop: mu_indeps_conv_comb}
$X\in\mI_1\cup\mI_2$ is a $b$-independent space of $\mM$, i.e., $\mI_1\cup\mI_2\subseteq\mI_b(\mM)$.
\end{proposition}
\begin{proof}
Let $\rho$ be the $q$-rank function of $\mM$. We distinguish different cases. First assume that $X\in\mI_1\cap\mI_2$, then we know that $\rho_1(Y)=\rho_2(Y)=\dim(Y)$ for every $Y\leq X$. Thus
$$\rho(Y)=\frac{a}{b}\rho_1(Y)+\left(1-\frac{a}{b}\right)\rho_2(Y)=\dim(Y),
$$
and we get $\rho(Y)\geq\frac{\dim(Y)}{b}$. Therefore $X$ is $b$-independent. Next assume that $X\in\mI_1$ but $X\not\in\mI_2$. This means that we have $\rho_1(Y)=\dim(Y)$ for all $Y\leq X$. Now since $\rho_2(Y)\geq 0$ we compute
$$\rho(Y)=\frac{a}{b}\rho_1(Y)+\left(1-\frac{a}{b}\right)\rho_2(Y)>\frac{a}{b}\rho_1(Y)=\frac{a}{b}\dim(Y)\geq\frac{\dim(Y)}{b},
$$
where the last inequality follows since $a\geq 1$.
Thus, also in this case, $X$ is $b$-independent. Finally, for the case $X\in\mI_2$ but $X\not\in\mI_1$ we can argue in the same way as above.
\end{proof}

\begin{remark}\label{rem: mu_indep_conv_comb}
    Note that there exist $\mu$-independent spaces in $\mM$ which are neither independent in $\mM_1$ nor in $\mM_2$. However, these spaces are not easy to classify, since the $\mu$-independence in these cases highly depends on the denominator $\mu$, the ambient dimension $n$ and the ranks of~$\mM_1$ and $\mM_2$. We refer to \Cref{ex:unif_sum}, for an illustration of this property.
\end{remark}

The next corollary is a consequence of \Cref{prop: mu_indeps_conv_comb} and characterizes some properties of $\mu$-dependent spaces and $\mu$-circuits, for a general denominator $\mu$. We fix the notation $\mD_i$ for the collection of dependent spaces and $\C_i$ for the collection of circuits of $\mM_i$, for $i=1,2$.

\begin{corollary}
    Let $\mD_\mu(\mM)$ and $\C_\mu(\mM)$ be the collections of $\mu$-dependent spaces and $\mu$-circuits of $\mM$, respectively. Then it holds that
    \begin{enumerate}
        \item $\mD_\mu(\mM)\subseteq\mD_1\cap\mD_2$ and
        \item $(\C_1\cup\C_2)\cap\mD_\mu(\mM)\subseteq\C_\mu(\mM)$.
    \end{enumerate}
\end{corollary}
\begin{proof}
    The first statement follows by taking the complement of both sides of the expression $\mI_1\cup\mI_2\subseteq\mI_\mu(\mM)$ from \Cref{prop: mu_indeps_conv_comb}. For the second statement we have to consider three cases. Assume first that $X\in(\C_1\cap\C_2)\cap\mD_\mu$. So we know that $X$ is $\mu$-dependent and all its proper subspaces $Y<X$ are contained in $\mI_1\cap\mI_2$, thus they are all $\mu$-independent by \Cref{prop: mu_indeps_conv_comb}. This implies that $X$ is a $\mu$-circuit. Now we assume that $X\in\C_1\cap\mD_\mu(\mM)$ but $X\not\in\C_2$. Once more we get that $X$ is $\mu$-dependent, moreover we know that all its proper subspaces $Y<X$ are contained in $\mI_1$. Then by \Cref{prop: mu_indeps_conv_comb} they are $\mu$-independent and so $X$ is a $\mu$-circuit. Lastly, for the case $X\in\C_2\cap\mD_\mu(\mM)$ but $X\not\in\C_1$, we can argue as above.
\end{proof}

A more precise characterization of the collections $\mD_\mu(\mM)$ and $\C_\mu(\mM)$ would be more involved due to the issues addressed in \Cref{rem: mu_indep_conv_comb}.

\section{The convex combination of paving \emph{q}-matroids}\label{sec: paving_conx_comb}

In this section we focus on the convex combination of special $q$-matroids, namely those that are \emph{paving}. We recall that a $q$-matroid $\mM$ is said to be paving when every circuit $C$ satisfies $\dim(C)\geq\text{rank}(\mM)$; see \cite{luerssen2023representabilitydirectsum}. We generalize the \emph{pavingness} property for rational $q$-polymatroids. Afterwards, we show that any convex combination of paving $q$-matroids with the same ranks gives rise to paving $q$-polymatroids. Furthermore, we discuss the flats and cyclic spaces of such a convex combination.

We start with an explicit construction of paving $q$-matroids, which was proven in \cite{gluesing2022q}.

\begin{proposition}\cite[Prop. 4.6.]{gluesing2022q}\label{prop: paving_construction}
    Let $1\leq k\leq n-1$. Let $\mathcal{S}\subseteq\mL(E)_{k}$, such that for every two distinct $V,W\in\mathcal{S}$ we have $\dim(V\cap W)\leq k-2$. Define the map
    \[        \rho_\mS:\mL(E)\rightarrow\mathbb{Z}_{\geq 0},\quad V\mapsto
        \begin{cases}
            \;\hfil k-1\;&\text{if }V\in\mathcal{S},\\
            \;\hfil \min\{\dim V,k\}\;&\text{otherwise}. 
        \end{cases}
    \]
    Then $\mM_\mS=(\mL(E),\rho_\mS)$ is a paving $q$-matroid of rank $k$, whose circuits of rank $k-1$ are the subspaces in $\mathcal{S}$. We call $\mM_\mS$ the \textbf{paving $q$-matroid induced by $\mS$}.
\end{proposition}

\begin{remark}
    Not all paving $q$-matroids arise  from \Cref{prop: paving_construction}. We will restrict our study only to this class, since we know exactly the ``shape" of their corresponding points in the $q$-rank polytope $\RP$.
\end{remark}

A special example of paving $q$-matroids constructed as in \Cref{prop: paving_construction} is given below.

\begin{example}
    Let $\mS=\emptyset$ and $1\leq k\leq n-1$. Then the $q$-rank function from \Cref{prop: paving_construction} simplifies to $\rho(V)=\min\{\dim V,k\}$, for all $V\in\mL(E)$. This is precisely the uniform $q$-matroid $\mU_{k,n}(q)$, having all the $(k+1)$-dimensional subspaces of $\F_q^n$ as circuits. So, we can regard the uniform $q$-matroid $\mU_{k,n}(q)$ as $\mM_\emptyset$, i.e. the paving $q$-matroid induced by $\emptyset$. 
\end{example}

The next definition introduces the notion of pavingness for the more general case of rational $q$-polymatroids.

\begin{definition}\label{def: mu_paving}
    Let $\mM$ be a rational $q$-polymatroid with ground space $E$ and denominator $\mu$. Then we say that $\mM$ is \textbf{$\mu$-paving} if for all $\mu$-circuits $C\in\mL(E)$ it holds that $\dim(C)\geq \max\{\dim(V)\;\mid\; V\in\mI_\mu(\mM)\}$.
\end{definition}

\begin{remark}
    We want to emphasize that this notion of pavingness is not a simple generalization of the same definition for $q$-matroids. The reason is that the rank of a $q$-polymatroid is given by the rank of its $\mu$-bases, which are the maximum-dimensional $\mu$-independent spaces. However, this rank is not necessarily equal to their dimension, which implies
    \[
        \max\{\dim(V)\;\mid\; V\in\mI_\mu(\mM)\}\geq\text{rank}(\mM).
    \]
    Therefore, if a $q$-polymatroid is paving according to \Cref{def: mu_paving}, then the dimension of every $\mu$-circuit is greater than or equal to the rank of the $q$-polymatroid. A detailed discussion about $\mu$-independent spaces can be found in \cite[Secs. 3--4]{gluesing2022independent}.       
\end{remark}

For the remainder of the section, we fix the following setting. Let $2\leq k\leq n-1$ and $\mS_1,\mS_2\subseteq \mL(E)_{k}$, with the property described in \Cref{prop: paving_construction}. Assume for simplicity that $\mS_1\cap\mS_2=\emptyset$ (we will address the other case later on). For $i=1,2$ let $\mM_i=\mM_{\mS_i}$ be the paving $q$-matroid induced by $\mS_i$ and $p_{\mM_i}$ be its corresponding point in $\RP$. Moreover, let $\mM = \frac{a}{b}\mM_1 +\left(1-\frac{a}{b}\right)\mM_2$ be the convex combination of $\mM_1$ and $\mM_2$ with rational coefficient $\frac{a}{b}$, where $a\geq 1$ and $b>a$, as in \Cref{def: conv_combi} and consider the denominator $b$ for $\mM$. Finally, denote by $p_\mM\in\RP$ the corresponding point of $\mM$. 

The following theorem shows that the paving property is preserved under taking rational convex combinations.

\begin{theorem}\label{thm: mu_paving_conv_comb}
     $\mM = \frac{a}{b}\mM_1 +\left(1-\frac{a}{b}\right)\mM_2$ is $b$-paving.
\end{theorem}

Before proving \Cref{thm: mu_paving_conv_comb} we first collect some useful properties concerning the shape of the points $p_{\mM_1}, \; p_{\mM_2}$ and $p_\mM$.

\begin{remark}\label{rem: shape_paving_points}
    By \Cref{prop: paving_construction}, we immediately get that for $i=1,2$,
    \[
       p_{\mM_i}=(p_{\mM_i,X})_{X\in\mL(E)}=
       \begin{cases}
        \;\hfil k-1\;&\text{if }X\in\mS_i,\\
        \;\hfil \dim(X)\;&\text{if }\dim(X)\leq k\; \textnormal{and }X\not\in\mS_i,\\
        \;\hfil k\;&\text{otherwise}. 
        \end{cases} 
    \]
    Therefore the point $p_\mM$ corresponding to the convex combination of $\mM_1$ and $\mM_2$ is given by
    \[
        p_{\mM}=(p_{\mM,X})_{X\in\mL(E)}=
       \begin{cases}
        \;\hfil \dim(X)\;&\text{if }\dim(X)\leq k\; \textnormal{and } X\not\in\mS_1\cup\mS_2,\\
        \;\hfil \frac{a}{b}(k-1)+\left(1-\frac{a}{b}\right)k\;&\text{if } X\in\mS_1,\\
        \;\hfil \frac{a}{b} k+\left(1-\frac{a}{b}\right)(k-1)\;&\text{if } X\in\mS_2,\\
        \;\hfil k\;&\text{if } \dim(X)\geq k+1. 
        \end{cases}
    \]
\end{remark}

\begin{proof}[Proof of \Cref{thm: mu_paving_conv_comb}]
    Let $\mI_1,\mI_2$ be the collections of independent spaces of $\mM_1$ and $\mM_2$ respectively and set $\mI_b(\mM)$ to be the collection of $b$-independent spaces of $\mM$. By \Cref{rem: shape_paving_points} we know that every space $X\not\in\mS_i$ of dimension $\dim(X)\leq k$ is independent in $\mM_i$, for $i=1,2$. Thus we get that $\mI_1\cup\mI_2=\mL(E)_{\leq k}$, since $\mS_1\cap\mS_2=\emptyset$. Furthermore, all spaces $X$ having dimension $\dim(X)\leq k$, are $b$-independent in $\mM$, by \Cref{prop: mu_indeps_conv_comb}. Now, the spaces $X\in\mL(E)_s$ with $s\in\{k+1,\ldots,n\}$ have the same rank value $k$,  from \Cref{rem: shape_paving_points}. In particular, for such $s\in\{k+1,\ldots,n\}$,  either every $X\in\mL(E)_s$ satisfies
    \begin{equation}\label{eq:mu-circuit-paving}
        p_{\mM,X}\geq\frac{\dim(X)}{b},
    \end{equation}
    or none of them does. In the first case, every subspace $Y$ such that $\dim(Y)\in\{k+1,\ldots,s-1\}$ would also satisfy inequality \eqref{eq:mu-circuit-paving} and, therefore, we would get
    \[ \bigcup_{t=k+1}^s\mL(E)_t\subseteq\mI_b(\mM).
    \]
    In the second case, we would declare all of those spaces $X$ to be $b$-dependent. All together this implies that the $b$-circuits are given by the collection of all $s_0$-dimensional spaces such that $s_0\in\{k+1,\ldots,n\}$ is the smallest dimension for which the inequality \eqref{eq:mu-circuit-paving} is violated the first time. Moreover, 
    $$\mI_b(\mM)=\bigcup_{t=0}^{s_0-1}\mL(E)_t.         
    $$
    Thus, the maximal dimension of any $b$-independent space is $s_0-1<s_0$ and so $\mM$ is $b$-paving according to \Cref{def: mu_paving}.
\end{proof}

As a byproduct, we obtain the collections of the $b$-independent spaces and $b$-circuits of $\mM$.

\begin{corollary}
    Let $s_0\in\{k+1,\ldots,n\}$ be the smallest integer such that there exists a space $X\in\mL(E)_{s_0}$ with $p_{\mM,X}<\dim(X)/b$. Then the following holds:
    \[
        \mI_b(\mM)=\mL(E)_{\leq s_0-1} \quad \text{ and } \quad \C_b(\mM)=\mL(E)_{s_0}. 
    \]
\end{corollary}

\begin{remark}
    \Cref{thm: mu_paving_conv_comb} still holds if we replace the assumption $\mS_1\cap\mS_2=\emptyset$ with the weaker one $\mS_1\not=\mS_2$. The reason is that the additional possibility $p_{\mM,X}=k-1$ for $X\in\mS_1\cap\mS_2$ also yields that $X$ is $b$-independent. This is due to the fact that $b\geq 2$, while, at the same time, $\frac{k}{k-1}<2$ and consequently we have $k-1\geq k/b$. Therefore, all spaces of dimension less than or equal to~$k$ are $b$-independent in $\mM$ and we arrive again at the starting point of the proof of \Cref{thm: mu_paving_conv_comb}.
\end{remark}

For the remainder of this subsection we focus on the characterization of the flats, cyclic spaces and cyclic flats of a convex combination $\mM$ arising from two paving $q$-matroids $\mM_1$ and $\mM_2$. We start by describing these collections for general $q$-matroids coming from the construction in \Cref{prop: paving_construction}.

\begin{lemma}\label{lem: flats_cyclic_paving_q_mats}
    Let $1\leq k\leq n-1$ and $\mS\subseteq \mL(E)_{k}$ with the property described in \Cref{prop: paving_construction}, and let $\mM_\mS$ be the paving $q$-matroid induced by $\mS$. Then the following holds:
    \begin{enumerate}
        \item The flats of $\mM_S$ are given by
        \[
            \mF(\mM_{\mS})=(\mL(E)_{\leq k-1}\setminus I)\cup \mS\cup\{\F_q^n\},
        \]
        where $I$ is given by
        \[
            I:=\{H\in\mL(E)_{k-1}\;\mid\; H \text{ is contained in some } S\in\mS\}.
        \]
        
        \item The cyclic spaces of $\mM_S$ are given by 
        \[
            \mO(\mM_{\mS})=(\mL(E)_{\geq k+1}\setminus J)\cup \mS\cup\{\langle 0\rangle\},
        \]
        where $J$ is given by
        \[
            J:=\{Y\in\mL(E)_{k+1}\;\mid\; Y \text{ contains some } S\in\mS\}.
        \]
        
        \item The cyclic flats are given by
        \[
            \mZ(\mM_{\mS})= \mS\cup\{E\}\cup\{\langle 0\rangle\}.
        \]
    \end{enumerate}
\end{lemma}
\begin{proof}
    (1) By \Cref{rem: shape_paving_points}, we have that each space $X\in\mL(E)_{\geq k+1}$  has the same rank $k$, thus it cannot be a flat of $\mM_{\mS}$, except for the whole space $E$. Next, we have that all $k$-dimensional spaces are either bases of $\mM_{\mS}$ or members of $\mS$. While bases can never be flats, all the members of $\mS$ are flats, since their rank increases from $k-1$ to $k$ whenever we add any  $1$-dimensional space not already contained in them. Let $X\in\mL(E)_{k-1}$. Then, $X$ is independent. 
    If $X$ is not contained in any element of $\mS$, then it is clearly a flat of $\mM_{\mS}$. If $X$ is an hyperplane of some element of $\mS$, we observe that there exists a $1$-dimensional space that, added to $X$, gives a space of the same rank as $X$. Hence, $X$ cannot be a flat. Finally, all the spaces $X\in\mL(E)_{\leq k-2}$ are independent and so are all the $(\dim(X)+1)$-dimensional spaces containing them, thus they are flats of $\mM_{\mS}$. 

    (2) Cyclic spaces cannot be independent, except from the zero-space. So, we only need to consider dependent spaces. Clearly, circuits are cyclic, thus all elements of $\mS$ are cyclic in $\mM_{\mS}$. Let $X\in\mL(E)_{k+1}$. We distinguish two cases. If $X$ does not contain any element from $\mS$, then it is cyclic, by \Cref{rem: shape_paving_points}. Assume that $X$ contains an element from $S\in\mS$. Then $S$ has rank strictly smaller than the rank of $X$, by \Cref{rem: shape_paving_points}, and therefore $X$ cannot be cyclic in $\mM_{\mS}$. Finally, all spaces $X\in\mL(E)_{\geq k+2}$ have rank $k$ and so do all their hyperplanes, which implies that they are cyclic in $\mM_{\mS}$. All together this proves the second statement.

    (3) This statement follows immediately by intersecting the sets $\mF(\mM_{\mS})$ and $\mO(\mM_{\mS})$.
\end{proof}

With the aid of \Cref{lem: flats_cyclic_paving_q_mats} we can characterize the flats, cyclic spaces and cyclic flats of the convex combination of two paving $q$-matroids. Therefore let $\mM_1,\mM_2$ and $\mM$ be as before.

\begin{theorem}\label{thm: flats_paving_conv_comb}
   The collection of flats of $\mM$ is given by
    \[
        \mF(\mM)=\mS_1\cup\mS_2\cup\{E\}\cup\mL(E)_{\leq k-1}.
    \]   
\end{theorem}
\begin{proof}
 Let $\mF_i$ denote the collection of flats of $\mM_i$, for $i=1,2$. For any $X\in\mL(E)$ let $\text{Cl}_1(X)$ and $\text{Cl}_2(X)$ the sets defined in \Cref{def-closure} corresponding to $\mM_1$ and $\mM_2$, respectively. By \Cref{coro:flats_conv_comb}, we already have $\mF_1\cup\mF_2\subseteq\mF_{\mM}$. Using \Cref{lem: flats_cyclic_paving_q_mats} (1) for $\mF_1$ and $\mF_2$ individually, one computes the following:
    \begin{align*}
        \mF_1\cup\mF_2 &= \mS_1\cup\mS_2\cup\{E\}\cup(\mL(E)_{\leq k-1}\setminus I_1)\cup(\mL(E)_{\leq k-1}\setminus I_2)\\
        &= \mS_1\cup\mS_2\cup\{E\}\cup(\mL(E)_{\leq k-1}\setminus (I_1\cap I_2)),
    \end{align*}
    where $I_1$ and $I_2$ denote the corresponding sets from \Cref{lem: flats_cyclic_paving_q_mats}(1). For all the other spaces $F\in\mL(E)\setminus(\mF_1\cup\mF_2)$ we still need to check if they are flats of $\mM$. Due to \Cref{prop:flats_conv_comb}, this just means to check if $\text{Cl}_1(F)\cap\text{Cl}_2(F)=\mL(F)_1$ holds. There exist three cases to consider.
    \begin{itemize}
        \item $F\in\mL(E)_k\setminus(\mS_1\cup\mS_2)$: This implies that $F$ is basis of $\mM_1$ and $\mM_2$ and, consequently, we get $\text{Cl}_1(F)=\text{Cl}_2(F)=\mL(E)_1$. Thus, $F$ is not a flat of $\mM$.

        \item $F\in\mL(E)_{\geq k+1}\setminus\{\F_q^n\}$: In this case we know from \Cref{rem: shape_paving_points} that $p_{\mM_1,F}=p_{\mM_2,F}=k$ and  this is also the rank of all the spaces containing $F$. Hence, $\text{Cl}_1(F)=\text{Cl}_2(F)=\mL(E)_1$ and $F$ is not a flat of $\mM$.

        \item $F\in I_1\cap I_2$: This means that there exist spaces $X_1\in\mS_1$ and $X_2\in\mS_2$ such that $F\leq X_1\cap X_2$. Moreover $X_1\not=X_2$ since $\mS_1\cap\mS_2=\emptyset$ by assumption. Now, if there exists a $1$-dimensional space $x\in\mL(E)_1\setminus\mL(F)_1$ such that $F+x=X_1$ and $F+x=X_2$, we get that $X_1=X_2$, which contradicts the assumption. Therefore, none of the $1$-dimensional spaces $x\in\mL(E)_1\setminus\mL(F)_1$ can be in $\text{Cl}_1(F)\cap\text{Cl}_2(F)$, implying that $\text{Cl}_1(F)\cap\text{Cl}_2(F)=\mL(F)_1$ and hence $F$ is a flat of $\mM$.
    \end{itemize}
    By combining these three cases with the above computation of $\mF_1\cup\mF_2$, we obtain the desired result. 
\end{proof}

\begin{theorem}\label{thm: cyclic_paving_conv_comb}
    The collection of cyclic spaces of $\mM$ is given by
    \[
        \mO(\mM)=\mS_1\cup\mS_2\cup\{\langle 0\rangle\}\cup\mL(E)_{\geq k+1}.
    \] 
\end{theorem}
\begin{proof}
Let $\mO_i$ denote the collection of cyclic spaces of $\mM_i$, for $i=1,2$.  Notice that $\mM_1$ and $\mM_2$ do not have loops, since they are paving. Therefore, we can apply \Cref{coro:cyc_spaces_noloops}, which yields $\mO_1\cup\mO_2\subseteq\mO(\mM)$. Then, using \Cref{lem: flats_cyclic_paving_q_mats} (2) for $\mF_1$ and $\mF_2$ individually, one computes the following:
    \begin{align*}
        \mO_1\cup\mO_2 &= \mS_1\cup\mS_2\cup\{\langle 0\rangle\}\cup(\mL(E)_{\geq k+1}\setminus J_1)\cup(\mL(E)_{\geq k+1}\setminus J_2)\\
        &= \mS_1\cup\mS_2\cup\{\langle 0\rangle\}\cup(\mL(E)_{\geq k+1}\setminus (J_1\cap J_2)),
    \end{align*}
    where $J_1$ and $J_2$ denote the corresponding sets from \Cref{lem: flats_cyclic_paving_q_mats} (2). For all the other spaces $V\in\mL(E)\setminus(\mO_1\cup\mO_2)$ we still need to check if they are cyclic spaces of $\mM$. Due to \Cref{prop:cyc_spaces_noloops}, this just means to check if $\mN_1(V)\cap\mN_2(V)=\emptyset$ holds, where $\mN_1$ and $\mN_2$ are the sets defined in \Cref{subsec: cyclic}. We need to distinguish between three cases.
    \begin{itemize}
        \item $V\in\mL(E)_k\setminus(\mS_1\cup\mS_2)$: This implies that $V$ is basis of both $\mM_1$ and $\mM_2$ and consequently we get $\mN_1(V)=\mN_2(V)=\mH(V)$. Thus $V$ is not a cyclic space of $\mM$.

        \item $V\in\mL(E)_{\leq k-1}\setminus\{\langle 0\rangle\}$: In this case we know that $V$ is independent in $\mM_1$ and $\mM_2$, hence we have again $\mN_1(V)=\mN_2(V)=\mH(V)$ and $V$ is not a cyclic space of $\mM$.

        \item $V\in J_1\cap J_2$: This means that there are some hyperplanes of $V$ contained in $\mS_1$ and some contained in $\mS_2$. Moreover, by  definition of $\mN_i$ we can observe that $\mN_i(V)\subseteq\mS_i$, for $i=1,2$. Thus, there is no hyperplane of $V$ contained in $\mN_1(V)\cap\mN_2(V)$, since $\mS_1\cap\mS_2=\emptyset$. In other words, we have $\mN_1(V)\cap\mN_2(V)=\emptyset$ and $V$ is a cyclic space of $\mM$. 
    \end{itemize}
    By combining these three cases with the above computation of $\mO_1\cup\mO_2$, we obtain the desired result.
\end{proof}

As a direct consequence of the above theorems we get the following description of the cyclic flats of $\mM$.

\begin{corollary}\label{coro: cyclic_flats_paving_conv_comb}
   The collection of cyclic flats of $\mM$ is given by
    \[
        \mZ(\mM)=\mS_1\cup\mS_2\cup\{\langle 0\rangle\}\cup\{E\}.
    \]
\end{corollary}
\begin{proof}
The result follows by intersecting the sets $\mF(\mM)$ and $\mO(\mM)$ obtained from Theorems~\ref{thm: flats_paving_conv_comb} and~\ref{thm: cyclic_paving_conv_comb}.
\end{proof}

Note that by weakening the assumption $\mS_1\cap\mS_2=\emptyset$ in Theorems~\ref{thm: flats_paving_conv_comb} and~\ref{thm: cyclic_paving_conv_comb} and \Cref{coro: cyclic_flats_paving_conv_comb} to the condition $\mS_1\ne\mS_2$, the computation of these collections becomes more complicated and the results will not have such a compact form anymore.

We end this subsection with one final observation.

\begin{remark}
    The reader may have noticed that our setting only considers two paving $q$-matroids coming from the construction in \Cref{prop: paving_construction}, which in addition have the same rank. One can also do it for two paving $q$-matroids of arbitrary rank, still coming from construction in \Cref{prop: paving_construction}. However, all the computations become significantly more difficult, as the shape of the points $p_{\mM_1}, p_{\mM_2}$ and $p_{\mM}$ include several new cases to consider. For this reason, we decided to not discuss this situation in general. However, for uniform $q$-matroids we obtain interesting results, see \Cref{sec: uniform_conv_comb}.
\end{remark}


\section{The convex combination of uniform \emph{q}-matroids}\label{sec: uniform_conv_comb}


In this section, we first describe the convex combination with rational coefficients of two uniform $q$-matroids of any rank, by investigating its $\mu$-independent spaces. Then, we turn to the convex combination of $n-2$ uniform $q$-matroids with rank strictly increasing. Finally, we characterize the cyclic flats of such convex combinations.

Recall that if $\mM$ is the uniform $q$-matroid $\mU_{k,n}(q)$, then its corresponding point $p_\mM\in\RP$ is 
\begin{equation}\label{eq:uniform_point}
    p_\mM=(0,1,\ldots,1,2,\ldots,2,\ldots,k,\ldots,k,\dots,k),
\end{equation}
where the first coordinate is $0$, the successive $\qqbin{n}{1}_q$ coordinates are equal to $1$, and so on, till the last $\sum\limits_{i=k}^n\qqbin{n}{i}_q$ positions are equal to $k$.

\subsection{Convex combination of two uniform \emph{q}-matroids}


In this subsection, we consider two uniform $q$-matroids. Let $1< k_1<k_2<n$ and
$\mM_1=\mU_{k_1,n}(q)$, $\mM_2=\mU_{k_2,n}(q)$. Then, $p_{\mM_1}$ and $p_{\mM_2}$ are of the form described in Eq. \eqref{eq:uniform_point}.
Let $\lambda=\frac{a}{\mu}\in\mathbb{Q}$, with $a\geq 1$ and $\mu\geq 2$. Consider the convex combination $\mM=(1-\lambda)\mM_1+\lambda\mM_2$ and its corresponding point $p_\mM\in\RP$. We have already seen that $\mu$ is a denominator for $\mM$. Hence, we can study the $\mu$-independence.

\begin{theorem}\label{thm:mu-indep_uniform}
    Let $a,\mu$ be integers, with $a\geq 1, \; \mu\geq 2$ and $\lambda=\frac{a}{\mu}\in\mathbb{Q}$ and $\mM=(1-\lambda)\mM_1+\lambda\mM_2$. If $\mu\geq \left\lceil \frac{n}{k_1}\right\rceil$, then $\mI_\mu(\mM) = \mL(E)$.
\end{theorem}
\begin{proof}
     Let $p_\mM=(p_X)_{X\in\mL(E)}$ be the point in $\RP$ corresponding to $\mM$. Then 
    $$p_X = \begin{cases}
        \dim(X) & \textnormal{ if }\dim(X)\leq k_1,\\
        k_1+\lambda(\dim(X)-k_1)   & \textnormal{ if } k_1<\dim(X)\leq k_2,\\
        (1-\lambda)k_1 + \lambda k_2 & \textnormal{ if } \dim(X)>k_2.
    \end{cases}$$
    From \Cref{prop: mu_indeps_conv_comb}, we have that the independent spaces in $\mM_1$ and $\mM_2$ are $\mu$-independent spaces in $\mM$. In particular, all the spaces $X\in\mL(E)_{\leq k_2}$ are $\mu$-independent. Let $X\in\mL(E)_{>k_2}$. Then $p_X = k_1+\lambda(k_2-k_1)$. We need to show that for every subspace $J\leq X$, we have that $p_J\geq \frac{\dim(J)}{\mu}$. This is clearly true for subspaces of dimension at most $k_2$. So, we consider $J\leq X$ with $\dim(J)> k_2$. Then 
    \begin{align*}
        p_J &= (1-\lambda)k_1 + \lambda k_2 = k_1 + \lambda(k_2-k_1) = k_1 + \frac{a}{\mu}(k_2-k_1) \\
        &\geq k_1 +\frac{k_2-k_1}{\mu} = \frac{(\mu-1)k_1+k_2}{\mu},
    \end{align*}
    where the inequality follows from the fact that $a\geq 1$. 
    Now, we have that 
    \begin{align}\label{eq:p1}
        \frac{(\mu-1)k_1+k_2}{\mu}\geq \frac{\dim(J)}{\mu} &\Leftrightarrow (\mu-1)k_1+k_2\geq \dim(J)\Leftrightarrow \frac{\dim(J)-k_2}{k_1} +1\leq \mu.
    \end{align}
    Observe that, since $k_1<k_2$, we get 
    \begin{align}\label{eq:p2}
        \frac{\dim(J)-k_2}{k_1} +1\leq \frac{\dim(J)-k_1}{k_1} +1 = \frac{\dim(J)}{k_1}\leq \left\lceil \frac{n}{k_1}\right\rceil \leq \mu,
    \end{align}
    where the last inequality follows by assumption. 
    Hence, by putting together Eq.~\eqref{eq:p1} and Eq.~\eqref{eq:p2}, we conclude that $X$ is $\mu$-independent. 
\end{proof}

The assumption $\mu\geq \left\lceil \frac{n}{k_1}\right\rceil$ in \Cref{thm:mu-indep_uniform} is quite strong, but there are some cases where we can show the same result for every choice of $\lambda$, as the following theorem illustrates. 

\begin{theorem}\label{thm: mu_indep_uniform2}
 Let $1< k_1<k_2<n$, with $k_1+k_2\geq n$ and
let $\mM_1=\mU_{k_1,n}(q)$, $\mM_2=\mU_{k_2,n}(q)$. Let $\lambda=\frac{a}{\mu}\in\mathbb{Q}$ and $\mM=(1-\lambda)\mM_1+\lambda\mM_2$. Then $\mI_\mu(\mM) = \mL(E)$.
\end{theorem}
\begin{proof}
    Let $p_\mM=(p_X)_{X\in\mL(E)}$ be the point in $\RP$ corresponding to $\mM$. As in the proof of \Cref{thm:mu-indep_uniform}, it is enough to show that spaces of dimension larger than $k_2$ are $\mu$-independent. Let $X\in\mL(E)_{>k_2}$.  Then $p_X = k_1+\lambda(k_2-k_1)$. We need to show that for every subspace $J\leq X$, we have that $p_J\geq \frac{\dim(J)}{\mu}$. This is clearly true for subspaces of dimension at most $k_2$. So, we consider $J\leq X$ with $\dim(J)> k_2$. Then 
    \begin{align*}
        p_J &= \left(1-\frac{a}{\mu}\right)k_1 + \frac{a}{\mu} k_2 = \frac{\mu-a}{\mu}k_1 + \frac{a}{\mu}k_2.
    \end{align*} 
    Observe that since $a\geq 1$ and $\mu\geq 2$,
    \begin{align*}
        (\mu-a)k_1+ak_2 = \mu k_1 +a(k_2-k_1) \geq 2k_1 + k_2-k_1 = k_1+k_2 \geq n \geq \dim(J).
    \end{align*}
    By dividing by $\mu$ both sides of the inequality, we have that $X$ is $\mu$-independent.
\end{proof}

Note that while the assumptions of \Cref{thm:mu-indep_uniform} and \Cref{thm: mu_indep_uniform2} are sufficient to characterize the $\mu$-independent spaces of the convex combination of two uniform $q$-matroids, they are not necessary, as the next example shows.

\begin{example}\label{ex:unif_sum}
    Let $k_1=2, \; k_2=3, \; n=8$ and $\mM_1=\mU_{2,8}(q)$,   $\mM_2=\mU_{3,8}(q)$. Let $\lambda=\frac{2}{3}$. Then $\mu=3< 4=\frac{n}{k_1}$. Moreover, for every subspace $X$ of $\F_q^8$ with $\dim(X)\in\{4,5,6,7,8\}$, we have 
    $p_X = \frac{1}{3}\cdot 2 +\frac{2}{3}\cdot 3 = \frac{8}{3}\geq \frac{\dim(X)}{3}$, where $(p_X)_{X\in\mL(E)}$ is the point in $\RP$ corresponding to $\mM=\frac{1}{3}\mM_1 + \frac{2}{3}\mM_2$. Hence, all subspaces of $\F_q^n$ are $\mu$-independent in $\mM$, but $\mu<\frac{n}{k_1}$ and $n> k_1+k_2$.
\end{example}

\begin{remark}
    If $\mM_1=\mU_{k_1,n}(q)$ and $\mM_2=\mU_{n,n}(q)$, then it is immediate to see that for every $\lambda\in\mathbb{Q}$ and every $k_1<n$, we have $\mL(E)= \mI_\mu(\mM)$, where $\mM=(1-\lambda)\mM_1+\lambda\mM_2$.
\end{remark}

\subsection{Convex combination of \emph{n-2} uniform \emph{q}-matroids}
In this subsection, we extend previous results to the case of the convex combination of $n-2$ uniform $q$-matroids with rank values equal to $k_i=i+1$, for every $i\in\{1,\ldots,n-2\}$, with $n\geq 5$. Note that this cannot be done using associativity, since the convex combination of $q$-matroids is never a $q$-matroid.

For every $i\in\{1,\ldots,n-2\}$ let $\mM_i=\mU_{k_i,n}(q)$, $\lambda_i=\frac{a_i}{b_i}\in\mathbb{Q}$, with $0<\lambda_i<1$ and $\sum\limits_{i=1}^{n-2}\lambda_i=1$. Let $\mM = \sum\limits_{i=1}^{n-2}\lambda_i\mM_i$ and let $p_\mM=(p_X)_{X\in\mL(E)}$.

Observe that $\mI(\mM_i)\subseteq \mI(\mM_j)$ for every $1\leq i <j\leq n-2$. Moreover, it is not difficult to write down the explicit coordinates of $p_\mM$. Indeed, for $X\in\mL(E)$, we have that

$$p_X=\begin{cases}
    \dim(X) & \textnormal{ if } \dim(X)\leq 2, \\
    \sum\limits_{i=1}^{\dim(X)-2}(i+1)\lambda_i + \dim(X)\sum\limits_{i=\dim(X)-1}^{n-2}\lambda_i & \textnormal{ if } 3\leq \dim(X)\leq n-1, \\
    \sum\limits_{i=1}^{n-2}(i+1)\lambda_i & \textnormal{ if } \dim(X)=n.
\end{cases}$$

Let $\mu=\mathrm{lcm}(b_1,\ldots,b_{n-2})$. Then $\mu$ is a denominator for $\mM$ and $\mu\geq 3$.

\begin{theorem}\label{thm:mu_indep_flag_uniform}
$\mL(E)_{\leq n-1} \subseteq \mI_{\mu}(\mM)$.  Moreover, if $\mu\geq \left\lceil \frac{n}{2}\right\rceil$, then $\mL(E)=\mI_\mu(\mM)$.
\end{theorem}
\begin{proof}
  It is immediate that all the subspaces in $\mL(E)_{\leq 2}$ are $\mu$-independent, since they are strong independent spaces. Assume that $3\leq \dim(X)\leq n-1$. We show that $p_I\geq \frac{\dim(I)}{\mu}$ for all $I\leq X$ by using induction. Assume first that $\dim(X)=3$. Then
 $$p_X = 2\lambda_1 +3(1-\lambda_1)=3-\lambda_1>2,$$
 since $\lambda_1<1$. Moreover, since $\mu\geq 3$,  $\frac{\dim(X)}{\mu} = \frac{3}{\mu} \leq 1$. Hence we conclude that $p_X\geq \frac{\dim(X)}{\mu}$. For all other $I\leq X$, this is also clear, since they are $\mu$-independent. 

Assume that for all spaces $X\in\mL(E)_a$, for some $3<a<n-1$ we have that $p_I\geq \frac{\dim(I)}{\mu}$ for all $I\leq X$. 
In particular, the following holds:
\begin{align*}
    p_X &= \sum\limits_{i=1}^{a-2}k_i\lambda_i + a\sum\limits_{i=a-1}^{n-2}\lambda_i \\
    &= \sum\limits_{i=1}^{a-2}k_i\lambda_i + a\left(1-\sum\limits_{i=1}^{a-2}\lambda_i \right) \\
    &=\sum\limits_{i=1}^{a-2}(k_i-a)\lambda_i + a \geq \frac{a}{\mu}.
\end{align*}

We show that also for spaces of dimension $a+1$ the same property holds. Let $Y\in\mL(E)_{a+1}$. Then clearly for all $I\leq Y$, with $\dim(I)\leq a$, we have that $p_I\geq \frac{\dim(I)}{\mu}$. Moreover, we have that
\begin{align*}
    p_Y &= \sum\limits_{i=1}^{a-1}k_i\lambda_i + (a+1)\sum\limits_{i=a}^{n-2}\lambda_i \\
    &= \sum\limits_{i=1}^{a-1}k_i\lambda_i + (a+1)\left(1-\sum\limits_{i=1}^{a-1}\lambda_i \right) \\
    &=\sum\limits_{i=1}^{a-1}(k_i-a-1)\lambda_i + a+1\\
     &=\sum\limits_{i=1}^{a-2}(k_i-a-1)\lambda_i + (k_{a-1}-a-1)\lambda_{a-1}+ a+1\\
     &=\sum\limits_{i=1}^{a-2}(k_i-a)\lambda_i -\sum\limits_{i=1}^{a-2}\lambda_i -\lambda_{a-1}+ a+1\\
      &=\sum\limits_{i=1}^{a-2}(k_i-a)\lambda_i -\sum\limits_{i=1}^{a-1}\lambda_i + a+1\\
      &\geq \frac{a}{\mu} +1-\sum\limits_{i=1}^{a-1}\lambda_i\\
      &= \frac{a}{\mu} + \sum\limits_{i=a}^{n-2}\lambda_i\\
      &\geq \frac{a}{\mu} + \frac{1}{\mu} =\frac{a+1}{\mu}.
\end{align*}
This shows that every space of dimension at most $n-1$ is $\mu$-independent. We are left to show that if $\mu\geq \left\lceil \frac{n}{2}\right\rceil$, then $E$ is $\mu$-independent. Assume that $\bar{\lambda}=\min\{\lambda_i \; \mid \; i\in\{1,\dots, n-2\}\}$. Then $\bar{\lambda}=\frac{\bar{a}}{\bar{b}}$, for some $\bar{a},\bar{b}\in\mathbb{N}$. Then we have
\begin{align*}
    p_{E}&= \sum\limits_{i=1}^{n-2}k_i\lambda_i = \sum\limits_{i=1}^{n-2}(i+1)\lambda_i = \sum\limits_{i=1}^{n-2}i\lambda_i +1\\
    &\geq \sum\limits_{i=1}^{n-2}i\bar{\lambda} +1 = \frac{(n-1)(n-2)}{2}\bar{\lambda} +1 \\
     &\geq \sum\limits_{i=1}^{n-2}i\bar{\lambda} +1 = \frac{(n-1)(n-2)}{2}\frac{\bar{a}}{\mu} +1 \\
     &\geq \frac{n\bar{a}}{2\mu}+1 \geq \frac{n}{2\mu}+1.
\end{align*}
Since $\mu\geq\left\lceil\frac{n}{2}\right\rceil$, then we conclude that $p_{E}\geq \frac{n}{\mu}$.

\end{proof}

\subsection{Cyclic flats}

In this subsection, we illustrate the cyclic flats of the convex combination of two uniform $q$-matroids.

Let $1< k_1<k_2<n$, $\mM_1=\mU_{k_1,n}(q)$ and $\mM_2=\mU_{k_2,n}(q)$. In \cite[Proposition 3.30]{alfarano2024cyclic} it was shown that the cyclic flats of $\mM_i$ are $\langle 0 \rangle$ and $E$ and they have respectively rank values $0$ and~$k_i$, for $i=1,2$. Let $\mF_i$ and $\mO_i$ be respectively the collection of flats and cyclic spaces of $\mM_i$. 
It is clear that 
\begin{align*}
    \mF_i &= \mL(E)_{< k_i} \cup \F_q^n, \\
    \mO_i &= \mL(E)_{>k_i}\cup\langle 0 \rangle.
\end{align*}
Hence $\mF_1\subseteq \mF_2$ and $\mO_2\subseteq \mO_1$.
Let $\mF(\mM)$ and $\mO(\mM)$ be the sets of flats and cyclic spaces of $\mM$. 

\begin{lemma}\label{lem:flats_uniform}
    Let $\mM$ be any convex combination of $\mM_1$ and $\mM_2$. Then $\mF(\mM)=\mF_2$.
\end{lemma}
\begin{proof}
    Let $\rho_i$ be the rank function of $\mM_i$, for $i=1,2$. From \Cref{prop:flats_conv_comb}, we have that $\mF(\mM)=\{F\in\mL(E) \; \mid \; \Cl_{\rho_1}(F) \cap\Cl_{\rho_2}(F)  = \mL(F)_1\}$. Moreover, we have  $\mF_2 = \mF_1\cup \mF_2 \subseteq \mF(\mM)$. Let $X\notin \mF_2$, with $\dim(X)\geq k_2$. Hence, for every $x\in \mL(E)_1$ we have that $\rho_i(X) = \rho_i(X+x)$ for $i=1,2$. So, we conclude that $X$ cannot lie in $\mF(\mM)$.
\end{proof}

\begin{lemma}\label{lem:cyclic_uniform}
    Let $\mM$ be any convex combination of $\mM_1$ and $\mM_2$. Then $\mO(\mM)=\mO_1$.
\end{lemma}
\begin{proof}
    Let $\rho_i$ be the rank function of $\mM_i$, for $i=1,2$. Since $\mM_i$ has no loops, from \Cref{prop:cyc_spaces_noloops}, we have that $\mO(\mM)=\{A\in\mL(E) \; \mid \; \mN_1(A)\cap \mN_2(A)=\emptyset\}$, where $\mN_i(A)=\{H\in\mH(A) \; \mid \; \rho_i(H)<\rho_i(A)\}$. Moreover, from \Cref{coro:cyc_spaces_noloops}, we have that $\mO_1 = \mO_1\cup \mO_2 \subseteq \mO(\mM)$.  Let $A\in\mL(E)$, be a subspace with $1\leq \dim(A)\leq k_1$. Then $A$ is independent in both $\mM_i$, since $ k_1<k_2$. Then, for every $H\in\mH(A)$, $H$ is also independent and hence $\rho_i(H)<\rho_i(A)$ and $\mN_1(A) \cap \mN_2(A) \ne \emptyset$. So, we conclude that $A$ cannot lie in $\mO(\mM)$. 
\end{proof}

\begin{proposition}\label{prop:cyc_flats_uniform}
Let $\mM$ be any convex combination of $\mM_1$ and $\mM_2$. 
    The collection of cyclic flats of $\mM$ is given by $\mZ(\mM) := \{\langle 0 \rangle , E\} \cup \{A \in\mL(E) \;\mid\; k_1< \dim(A) < k_2\}$. 
\end{proposition}
\begin{proof}
    The collection of cyclic flats of $\mM$ is $\mZ(\mM) = \mF(\mM) \cap \mO(\mM) = \mF_2 \cap\mO_1$.
\end{proof}


\section{The characteristic Puiseux  polynomial}\label{sec: Puiseux_polyn}


In \cite[Sec.~3]{byrne2024weighted}, the characteristic polynomial of an integer $(q,r)$-polymatroid was introduced and studied. Moreover, the special case given by a $(q,1)$-polymatroid, i.e. a $q$-matroid, was investigated in more detail in \cite[Sec.~5]{jany2023proj_matroid}. In this section, we introduce a more general notion for rational $q$-polymatroids, inspired by \cite{byrne2024weighted}. However, this is not a polynomial anymore, but rather a truncated Puiseux series, i.e., a finite linear combination of monomials whose exponents may be rational and possibly negative.
Afterwards, we discuss some properties in the context of $q$-polymatroids arising from convex combinations of $q$-matroids.

Let us denote by $\mathbb{C}\{\!\{t\}\!\}$ the field of \emph{Puiseux series} with coefficients in $\mathbb{C}$, that is elements of the form
\[
    f=\sum_{k=k_0}^{+\infty}c_kt^{k/s},
\]
where $s$ is a positive integer and $k_0$ is an integer. In particular this means that we allow rational exponents of the indeterminate $t$, as long as all of these exponents have a bounded denominator and there exists a minimal one. We say that $f$ is a \emph{truncated Puiseux series} if we have only finitely many terms $c_k\neq 0$. 

The following definition is inspired by the concept of the characteristic polynomial given in \cite[Def. 23]{byrne2024weighted}.

\begin{definition}\label{def: char_Puiseux_series}
    Let $\mM=(\mL(E),\rho)$ be a rational $q$-polymatroid. We define the \textbf{characteristic Puiseux  polynomial of $\mM$} by
    \[
        \chi_\mM(t)=\sum_{X\in\mL(E)}\mu(\langle 0\rangle,X)t^{\ell(X)}\in\mathbb{C}\{\!\{t\}\!\}, 
    \]
    where, $\ell(X):=\rho(E)-\rho(X)$ for every $X\in\mL(E)$ and $\mu$ is the Möbius function, defined as 
    $$ \mu(\langle 0 \rangle, X)=(-1)^{\dim(X)}q^{\dim(X)\choose 2}.$$
\end{definition}

\begin{remark}\label{rem: char_Puiseux_series}
    From \Cref{def: char_Puiseux_series} we can immediately make the following observations:
    \begin{enumerate}
        \item[(a)] The characteristic Puiseux  polynomial $\chi_\mM$ of any rational $q$-polymaroid $\mM$ is indeed a truncated Puiseux series, due to the finiteness of $\mL(E)$. Its minimal exponent is always given by zero and its maximum exponent is $\rho(E)$. Furthermore, it has only integer coefficients.
        
        \item[(b)] When $\mM$ is a $(q,r)$-polymatroid for some $r\in\mathbb{N}$, the characteristic Puiseux polynomial reduces to the characteristic polynomial as defined in \cite[Def. 23]{byrne2024weighted}. Therefore our definition is a generalization of characteristic polynomial.
        
        \item[(c)] One can verify that given two isomorphic rational $q$-polymatroids $\mM_1$ and $\mM_2$, their characteristic Puiseux polynomials agree, i.e., $\chi_{\mM_1}(t)=\chi_{\mM_2}(t)$. Note that this constitutes a generalization of the result \cite[Lem. 25]{byrne2024weighted}, since the notion of isomorphism implies the notion of lattice-equivalence; see \cite[Def. 3]{byrne2024weighted}. Moreover the proof is identical to the one for \cite[Lem. 25]{byrne2024weighted}.

        \item[(d)] In general, every result8 in \cite[Sec. 3]{byrne2024weighted} about the characteristic polynomial of a $(q,r)$-polymatroid that only uses lattice-theoretic concepts, can be reproduced in a similar way for the characteristic Puiseux  polynomial of a rational $q$-polymatroid, for instance \cite[Lem. 26]{byrne2024weighted}. For all the results using concepts such as circuits, dependent spaces or independent spaces, we currently do not know if its possible to generalize them, using their rational $q$-polymatroid counterparts.      
    \end{enumerate}
\end{remark}

Now we want to examine rational $q$-polymatroids arising from rational convex combinations of two $q$-matroids.
We start with an explicit example.

\begin{example}\label{ex: char_Puiseux_series}
    Let $\mM_1=\mU_{2,3}(2)$ with $q$-rank function $\rho_1$ and $\mM_2=(\mL(\F_2^3),\rho_\mS)$ be the rank-$2$ paving $q$-matroid induced by the collection 
    \[
        \mS=\bigg\{\bigg\langle
        \begin{matrix}
            0&1&0\\
            0&0&1
        \end{matrix}\bigg\rangle\bigg\}\subset\mL(\F_2^3).
    \]
    Let $S$ be the only subspace in $\mS$.
    Consider the rational $q$-polymatroid $\mM=(\mL(\F_2^3),\rho)$ given by $\mM=\lambda\mM_1+(1-\lambda)\mM_2$ with rational coefficient $0<\lambda<1$. Then its characteristic Puiseux polynomial $\chi_\mM$ is given by the following computation:
    \begin{align*}
        \chi_\mM(t)&=\sum_{X\in\mL(\F_2^3)}\mu(\langle 0\rangle,X)t^{\ell(X)}\\
                   &= t^{\ell(\langle 0\rangle)} + \sum_{X\in\mL(\F_2^3)_1}(-1)t^{\ell(X)}+ \sum_{X\in\mL(\F_2^3)_2\setminus\mS}2t^{\ell(X)}+2t^{\ell(S)}-8t^{\ell(\F_2^3)}\\
                   &=t^2-7t^{1}+12t^{0}+2t^{1-\lambda}-8t^{0}\\
                   &=t^2-7t+4+2t^{1-\lambda}.
    \end{align*}
    
    Moreover, by \Cref{rem: char_Puiseux_series}, the characteristic polynomials $\chi_{\mM_1},\chi_{\mM_2}$ of the $q$-matroids $\mM_1$ and $\mM_2$, can also be easily computed using \Cref{def: char_Puiseux_series}. They are given as follows:
    \[
        \chi_{\mM_1}(t)=t^2-7t+6\quad\text{and}\quad\chi_{\mM_2}(t)=t^2-5t+4.
    \]

    Finally by comparing $\chi_\mM$ with these characteristic polynomials we observe that
     \[
        \chi_\mM(t)=\chi_{\mM_1}(t)-2+2t^{1-\lambda}=\chi_{\mM_1}(t)-\mu(\{0\},S)+\mu(\{0\},S)t^{1-\lambda}
    \]
    and similarly also
    \[
        \chi_\mM(t)=\chi_{\mM_2}(t)-2t+2t^{1-\lambda}=\chi_{\mM_2}(t)-\mu(\{0\},S)t+\mu(\{0\},S)t^{1-\lambda}.
    \]
    So, we see that the characteristic Puiseux polynomial of $\mM$ can be expressed in terms of either the characteristic polynomial of $\mM_1$ or $\mM_2$. 
\end{example}

The next results generalizes the observations of \Cref{ex: char_Puiseux_series} to arbitrary rational convex combinations of two paving $q$-matroids arising from \Cref{prop: paving_construction}, and having the same rank. We fix the following notation:

Let $1\leq k\leq n$ be an integer and $\mS_1,\mS_2\subseteq\mL(E)_k$ be two disjoint collections of $k$-dimensional spaces satisfying the condition from \Cref{prop: paving_construction}. Let $\mM_i$ be the paving $q$-matroid induced by $\mS_i$ and denote by $\chi_{\mM_i}$ its characteristic polynomial, for $i=1,2$. Set $\mM$ to be the convex combination of $\mM_1$ and $\mM_2$ with rational coefficient $0<\lambda<1$.

\begin{theorem}\label{prop: char_Puiseux_series_paving_conv_comb}
    The characteristic Puiseux polynomial of $\mM$ can be expressed in terms of the characteristic polynomials of either $\mM_1$ or $\mM_2$. In particular we get the following formulas:
    \begin{enumerate}
        \item In terms of $\chi_{\mM_1}$ we have:
        \[
            \chi_{\mM}(t)=\chi_{\mM_1}(t)+(-1)^kq^{k\choose 2}\Big(|\mS_1|(t^\lambda-t)+|\mS_2|(t^{1-\lambda}-1)\Big).
        \]

        \item In terms of $\chi_{\mM_2}$ we have:
        \[
            \chi_{\mM}(t)=\chi_{\mM_2}(t)+(-1)^kq^{k\choose 2}\Big(|\mS_1|(t^\lambda-1)+|\mS_2|(t^{1-\lambda}-t)\Big).
        \]
    \end{enumerate}
\end{theorem}
\begin{proof}
    Let $p_{\mM_1},p_{\mM_2}$ and $p_\mM$ be the points in $\RP$ corresponding to $\mM_1,\mM_2$ and $\mM$ respectively. Moreover, for all $X\in\mL(E)$ we denote by $\mu(X)$ the Möbius value $\mu(\langle 0\rangle,X)$. 
    For $i=1,2$ we have
    \[
       p_{\mM_i}=(p_{\mM_i,X})_{X\in\mL(E)}=
       \begin{cases}
        \;\hfil k-1\;&\text{if }X\in\mS_i,\\
        \;\hfil \dim(X)\;&\text{if }\dim(X)\leq k\; \textnormal{and } X\not\in\mS_i,\\
        \;\hfil k\;&\text{otherwise}. 
        \end{cases} 
    \]
    and hence
    \[
        p_{\mM}=(p_{\mM,X})_{X\in\mL(E)}=
       \begin{cases}
        \;\hfil \dim(X)\;&\text{if }\dim(X)\leq k\; \textnormal{and } X\not\in\mS_1\cup\mS_2,\\
        \;\hfil \lambda(k-1)+(1-\lambda)k\;&\text{if } X\in\mS_1,\\
        \;\hfil \lambda k+(1-\lambda)(k-1)\;&\text{if } X\in\mS_2,\\
        \;\hfil k\;&\text{if } \dim(X)\geq k+1. 
        \end{cases}
    \]
    Note that $p_{\mM_1,E}=p_{\mM_2,E}=p_{\mM,E}=k$ and furthermore $p_{\mM_1,X}=p_{\mM_2,X}=p_{\mM,X}$ for all $X\in\mL(E)\setminus(\mS_1\cup\mS_2)$. Therefore we can split the characteristic Puiseux polynomial of $\mM$ as follows:
    \begin{align}
        \chi_\mM(t)&=\sum_{X\in\mL(E)}\mu(X)t^{\ell(X)}\notag \\
                   &= \sum_{X\in\mL(E)\setminus(\mS_1\cup\mS_2)}\mu(X)t^{\ell(X)}+\sum_{X\in\mS_1}\mu(X)t^{\ell(X)}+\sum_{X\in\mS_2}\mu(X)t^{\ell(X)}.\label{eq: spilt_Puiseux}
    \end{align}
     We prove the first statement of the theorem. In Eq.\eqref{eq: spilt_Puiseux} we observe that for all $X\in\mL(E)\setminus(\mS_1\cup\mS_2)$ we have
    \[
        \ell(X)=k-p_{\mM_1,X}=\ell_1(X),
    \]
    where $\ell_1$ denotes the function from \Cref{def: char_Puiseux_series} corresponding to $\mM_1$. For all $X\in\mS_1$ we obtain
    \[
        \ell(X)=k-\lambda(k-1)-(1-\lambda)k=\lambda
    \]
    and finally for all $X\in\mS_2$ we see 
    \[
        \ell(X)=k-\lambda k-(1-\lambda)(k-1)=1-\lambda.
    \]
    Using these observations, we can rewrite Eq. \eqref{eq: spilt_Puiseux} in terms of the characteristic polynomial $\chi_{\mM_1}$ and further simplify this expression:
    \small{\begin{align*}
        \chi_\mM(t)&=\chi_{\mM_1}(t)-\Big(\sum_{X\in\mS_1}\mu(X)t^{\ell_1(X)}+\sum_{X\in\mS_2}\mu(X)t^{\ell_1(X)}\Big)+\sum_{X\in\mS_1}\mu(X)t^{\lambda}+\sum_{X\in\mS_2}\mu(X)t^{1-\lambda}\\
                   &=\chi_{\mM_1}(t)-(-1)^kq^{k\choose 2}|\mS_1|t-(-1)^kq^{k\choose 2}|\mS_2|+(-1)^kq^{k\choose 2}|\mS_1|t^\lambda+(-1)^kq^{k\choose 2}|\mS_2|t^{1-\lambda}\\
                   &=\chi_{\mM_1}(t)+(-1)^kq^{k\choose 2}\Big(|\mS_1|(t^\lambda-t)+|\mS_2|(t^{1-\lambda}-1)\Big),
    \end{align*}}
    where we used for the second equality that $\ell_1(X)=1$ for all $X\in\mS_1$, $\ell_1(X)=0$ for all $X\in\mS_2$ and finally that $\mu(X)=(-1)^kq^{k\choose 2}$ for all $X\in\mL(E)_k$. The second statement can be deduced with a similar argument by replacing $\chi_{\mM_1},\ell_1$ with $\chi_{\mM_2},\ell_2$. 
\end{proof}     

\begin{remark}\label{rem: char_Puiseux_series_paving_conv_comb}
    We want to emphasize that \Cref{prop: char_Puiseux_series_paving_conv_comb} does not say only that $\chi_{\mM}$ is determined by either $\chi_{\mM_1}$ or $\chi_{\mM_2}$. By using \cite[Thm. 5.8]{jany2023proj_matroid}, we also observe that $\chi_{\mM}$ can be computed by just using the flats of either $\mM_1$ or $\mM_2$. In other words, in order to determine $\chi_{\mM}$ we do not need any concepts associated to the $q$-polymatroid $\mM$, but rather only the well understood notion of flats of either of the $q$-matroids $\mM_1$ or $\mM_2$.          
\end{remark}

Note that the assumption $\mS_1\cap\mS_2=\emptyset$ is just for simplicity and \Cref{rem: char_Puiseux_series_paving_conv_comb} still holds if we omit this condition. However, in that case the resulting characteristic Puiseux polynomial of $\mM$ has a more intricate expression.

Finally, we note that situations like \Cref{prop: char_Puiseux_series_paving_conv_comb}, in which the characteristic Puiseux polynomial admits a particularly compact expression, appear to be rather rare. Nonetheless, there are still many open questions concerning the characteristic Puiseux polynomial of general rational $q$-polymatroids that we have not studied yet. We will mention some ideas for further research in \Cref{sec: final_remarks}. 


\section{Representable \emph{q}-polymatroids}\label{sec: rep_q_polymats}

In this section, we comment on the role of representable $q$-polymatroids in the polytope of all $q$-rank functions. One motivation for studying these objects is their link to rank-metric codes; see \cite{jurrius2018defining, gorla2019rank, shiromoto2019codes, gluesing2022q}. 

We start by briefly recalling some basic notions on rank-metric codes; for a more detailed treatment, we refer the reader to \cite{de78,gabidulin1985theory,gorla2021rank}.
For this purpose, we endow the space of matrices $\F_{q}^{n \times m}$ with the \textbf{rank distance}, defined by $\mathrm{d}(A,B):=\rk(A-B)$, for all $A,B\in\F_{q}^{n \times m}$.

\begin{definition}
 We say that $C \leq \F_q^{n \times m}$ is an {\bf $\F_q$-linear rank-metric code} or a {\bf matrix code} if $C$ is an $\F_q$-subspace of $\F_{q}^{n \times m}$. Its \textbf{minimum distance} is
 $$\mathrm{d}(C):=\min\{\rk(M) \;\mid\; M \in C, \; M \neq 0\}.$$ We say that $C$ is an $\F_q$-$[n \times m, k,d]$ rank-metric code if it has $\F_q$-dimension $k$ and minimum distance $d$. 
The \textbf{dual code} of $C$ is defined to be $C^{\perp}=\{M\in\F_q^{n\times m}: \mathrm{Tr}(MN^\top)=0 \textnormal{ for all } N\in~C\}$.
\end{definition} 

The parameters $n,m,k,d$ of a rank-metric code are related by the following inequality known as Singleton bound (see e.g. \cite{de78}):
$$k \leq \max\{n,m\}(\min\{n,m\}-d+1).$$
Codes whose parameters meet the Singleton bound with equality are called \textbf{maximum rank distance} codes or \textbf{MRD} for short.

It is known that an $\F_q$-linear subspace of $\F_q^{n\times m}$ induces a $q$-polymatroid; see \cite{gorla2019rank, shiromoto2019codes}. One way to describe this correspondence is as follows. Also for this section, we consider $E=\F_q^n$ and $\mL(E)$ the collection of all subspaces of $E$. For $U\in\mL(E)$, we denote by $U^\perp$ the orthogonal complement of $U$ with respect to the standard inner product. By a slight abuse of notation, we use the same symbol to denote both the orthogonal complement of a space and the dual of a code. This distinction should be clear from the context and should not cause confusion.
\begin{definition}\label{def:codepoly}
Let $C$ be an $\F_q$-$[n\times m,k]$ rank-metric code.
For each subspace $U\in\mL(E)$, we define
$$C(U):=\{M \in C \;\mid\; \mathrm{colsp}(M) \leq U^\perp\}.$$

It is immediate to see that for every $U\in\mL(E)$, $C(U)$ is an $\F_q$-subspace of $\F_q^{n\times m}$.
Let 
$$\rho: \mL (E) \longrightarrow \mathbb{Q}_{\geq 0}, \quad U\longmapsto\frac{k-\dim(C(U))}{m}.$$
Then $(\mL(E),\rho)$ is a $q$-polymatroid \cite[Theorem 5.3]{gorla2019rank} and we denote it by $\mM[C]$.	
\end{definition}

A $q$-polymatroid arising from a rank-metric code is said to be \textbf{representable}. In particular, this means that the rank function of representable $q$-polymatroids gives a rational point in the polytope $\RP$. 

Note that $\F_q^{n\times m}\cong \F_{q^m}^n$. If an $\F_q$-subspace of $\F_q^{n\times m}$ is in addition $\F_{q^m}$-linear, it gives rise to a $q$-matroid as follows.
	
\begin{definition}\label{def:vector_codes}
Let $C$ be a $k$-dimensional $\mathbb{F}_{q^m}$-linear subspace of $\F_{q^m}^n$. For every $W \in\mL(E)$, we define
$$C(W):=\{x \in C \;\mid\; \langle x_1, \ldots, x_n\rangle_{\F_q} \leq W^\perp\}.$$ 
Let $\rho: \mL (E) \longrightarrow \mathbb{Z}_{\geq 0}$ be defined by
$\displaystyle \rho(W):=k-\dim_{\F_{q^m}}(C(W)).$
Then $(\mL(E),\rho)$ is a $q$-matroid \cite[Thm. 24]{jurrius2018defining} and we also denote it by $\mM[C]$. Note that in this case, $k$ is the dimension of $C$ over $\F_{q^m}$. 	
\end{definition}

Given a rank-metric code $C\leq \F_q^{n\times m}$, in general it is not immediate to determine the dimension of $C(U)$, for $U\leq\F_q^n$. Indeed, we recall the following fact.

\begin{proposition}\cite[Prop. 6.2]{gorla2019rank}
    Let $C\leq \F_q^{n\times m}$ be a nonzero rank-metric code with minimum rank distance $d$ and let $d^\perp$ be the rank distance of $C^\perp$. Then for every $U\in\mL(E)$, 
    $$\rho_{\mM[C]}(U) = 
    \begin{cases}
        \frac{\dim(U)}{m} & \textnormal{ if } \dim(U)> n-d, \\
        \dim(U) & \textnormal{ if } \dim(U)<d^\perp.
    \end{cases}$$
\end{proposition}

There are some cases where we are able to say something more about the $q$-rank function of the $q$-polymatroid associated to a rank-metric code.

MRD codes constitute the most studied class of rank-metric codes. First of all, it is well known that MRD codes exist for every choice of parameters $n,m,d,q$. Furthermore, in \cite{gorla2019rank} it is shown that if $m\geq n$, then the $q$-polymatroid arising from an MRD
code is the uniform $q$-matroid of rank $n - d + 1$, where $d$ is the minimum rank distance of the code. More specifically, we have the following result.

\begin{proposition}\cite[Cor. 6.6]{gorla2019rank} Let $m\geq n$ and $C \leq \F_q^{n\times m}$ be an MRD code with minimum rank distance $d$. Let $\rho_\mM$ be the rank function of the $q$-polymatroid $\mM[C]$. Then for every $A\in\mL(E)$,
$$\rho_\mM(A) = \min\{n-d+1, \dim(A)\}.$$
Hence $\mM[C]$ is the uniform $q$-matroid $\mU_{n-d+1,n}(q)$.
    
\end{proposition}

We can see that this is not the case for $m<n$. Indeed, we recall the following result obtained in \cite{gluesing2022q}.

\begin{theorem}\cite[Theorem 3.10]{gluesing2022q}
    Let $m<n$ and $C \leq \F_q^{n\times m}$ be an MRD code with minimum rank distance $d$.  Then for every $U\in\mL(E)$,  we have
    $$\rho_{\mM[C]}(U) = 
    \begin{cases}
        \frac{n(m-d+1)}{m} & \textnormal{ if } \dim(U)\geq n-d+1, \\
        \dim(U) & \textnormal{ if } \dim(U)\leq m-d+1.
    \end{cases}$$
    Moreover, $\rho_{\mM[C]}(U)\geq \max\{1,\frac{\dim(U)}{m}\}(m-d+1)$ if $\dim(U)\in\{m-d+2, n-d\}$.
\end{theorem}
	
For $m=n-1$, the interval $\{m-d+2, n-d\}$ is empty for every value of the minimum distance~$d$. Hence, we have the following immediate corollary.

\begin{corollary}\label{cor:rank_function_special_MRD}
    Let $C \leq \F_q^{n\times (n-1)}$ be an MRD code with minimum rank distance $d$.  Then for every $U\in\mL(E)$, 
    $$\rho_{\mM[C]}(U) = 
    \begin{cases}
        \frac{n(n-d)}{n-1} & \textnormal{ if } \dim(U)\geq n-d+1, \\
        \dim(U) & \textnormal{ if } \dim(U)\leq n-d.
    \end{cases}$$
\end{corollary}

In particular, if $d>1$, $\mM[C]$ is not a $q$-matroid; see \cite[Cor. 3.11]{gluesing2022q}.

In the following result, we consider two representable $q$-polymatroids arising from rank-metric codes in $\F_q^{n\times (n-1)}$. Let $C_i$ be an $\F_q$-$[n\times m,k_i,d_i]$ MRD code, for $i=1,2$. Let $\mM_i=\mM[C_i]$ and let $\rho_i$ be its $q$-rank function, for $i=1,2$. Then, since $C_i$ is MRD, we have that $k_i=n(n-d)$ and, in particular, $\frac{k_i}{n}$ is an integer. The $q$-rank function $\rho_i$ reads as
$$ \rho_i(U)=\begin{cases}
    \dim(U) & \textnormal{ if } \dim(U)\leq \frac{k_i}{n},\\
    \frac{k_i}{n-1} & \textnormal{ if } \dim(U)> \frac{k_i}{n}.
\end{cases}$$

The next result generalizes in some way \Cref{thm: mu_indep_uniform2}. The proof follows the same idea as the one in \Cref{thm: mu_indep_uniform2}, but we include it for the sake of completeness.

\begin{theorem}\label{thm: mu_indep_poly}
    Assume that $1< k_1<k_2$ and $k_1+k_2\geq n$.  Let $\lambda=\frac{a}{b}\in\mathbb{Q}$, with $a,b\in\mathbb{N}$, $\gcd(a,b)=1$ and $\mM=(1-\lambda)\mM_1+\lambda\mM_2$. Then $\mu=b(n-1)$ is a denominator for $\mM$ and $\mI_\mu(\mM) = \mL(E)$.
\end{theorem}
\begin{proof}
     Let $p_\mM=(p_X)_{X\in\mL(E)}$ be the point in $\RP$ corresponding to $\mM$. Then 
    $$p_X = \begin{cases}
        \dim(X) & \textnormal{ if }\dim(X)\leq \frac{k_1}{n},\\
        \frac{k_1}{n-1}(1-\lambda)+\lambda\dim(X)   & \textnormal{ if } \frac{k_1}{n}<\dim(X)\leq \frac{k_2}{n},\\
        (1-\lambda)\frac{k_1}{n-1} + \lambda \frac{k_2}{n-1} & \textnormal{ if } \dim(X)>\frac{k_2}{n}.
    \end{cases}$$
    Let $\lambda=\frac{a}{b}\in\mathbb{Q}$, with $a,b\in\mathbb{N}$, $\gcd(a,b)=1$. Then, clearly, for $\mu=b(n-1)$ we have that $\mu p_X\in\mathbb{Z}$ for every $X\in\mL(E)$. Hence, it makes sense to consider $\mu$-independent spaces. First, observe that all the subspaces $X\in\mL(E)_{\leq \frac{k_1}{n}}$ are $\mu$-independent, since they are strong independent spaces. Let $X\in\mL(E)$ be such that $\frac{k_1}{n}<\dim(X)\leq \frac{k_2}{n}$. Then all its subspaces of dimension at most $\frac{k_1}{n}$ are $\mu$-independent. Let $J\leq X$ be such that $\dim(J)>\frac{k_1}{n}$. Then 
    \begin{align*}
        p_J &= \frac{k_1}{n-1} -\frac{ak_1}{b(n-1)} + \frac{a\dim(J)}{b} \\
        &=\frac{k_1 b + a\dim(J)(n-1) - ak_1}{\mu} \geq \frac{\dim(J)}{\mu},
    \end{align*}
    where the last inequality follows from the fact that $b>a\geq 1$. Finally, we consider a subspace $J\in\mL(E)$ with $\dim(J)>\frac{k_2}{n}$. Then,
    \begin{align*}
        p_J&= \frac{k_1(b-a)}{b(n-1)} + \frac{ak_2}{b(n-1)}\\
        &=\frac{k_1b + a(k_2-k_1)}{\mu}\\
        &\geq \frac{k_2 + k_1(b-1)}{\mu}\\
        &\geq \frac{k_1+k_2}{\mu} \\
        &\geq \frac{n}{\mu}\geq \frac{\dim(J)}{\mu},
    \end{align*}
    where we use the fact that $b>a\geq 1$ and $n\leq k_1+k_2$.
\end{proof}

\begin{remark}
Due to the existence of MRD codes, there are codes satisfying the assumption of \Cref{thm: mu_indep_poly}, i.e. $1<k_1<k_2$ and $k_1+k_2\geq n$. For example, let $n=5$, $k_1=10$ and $k_2=15$. Then clearly $C_1$ is an $\F_q$-$[5\times 4, 10,3]$ MRD code and $C_2$ is an $\F_q$-$[5\times 4, 15,2]$ MRD code.
\end{remark}

We want to conclude this section by observing that it is not possible to identify the points in $\RP$ that correspond to representable $q$-polymatroids. This is because there is no axiomatic system for those $q$-polymatroids. However, we might be inclined to speculate that they are vertices of $\RP$. Unfortunately, the following example shows that this is not the case. 

\begin{example}
    Let $C$ be an $\F_2$-$[3\times 2,3,2]_2$ MRD code and let $\mM[C]$ be the $q$-polymatroid arising from it. Then upon fixing an ordering on $\mL(\F_{2}^3)_i$, for every $0\leq i\leq n$, we have that the corresponding point in $\mathcal{P}_2^3$ is $(p_X)_{X\in\mL(\F_2^3)}=(0,1,1,1,1,1,1,1, 3/2,3/2,3/2, 3/2,3/2,3/2,3/2,3/2)$, since the $q$-rank function of $\mM[C]$ is of the type described in  \Cref{cor:rank_function_special_MRD}. We can verify with the aid of the computer algebra system \textsc{OSCAR} that this is not a vertex of $\mathcal{P}_2^3$. However, there exist vertices which are representable. For example, we observe that the point 
    $$(0,1/2, 1, 1, 1, 1/2, 1, 1, 3/2, 1, 1, 3/2, 3/2, 3/2, 3/2, 3/2)$$
    is a vertex and it can be represented as the $q$-rank function of the $q$-polymatroid arising from the $\F_2$-$[3\times 2,3,1]$ rank-metric code generated by the matrices
    $$\begin{pmatrix}
        1 & 0\\
        0 & 0\\
        0 & 0
    \end{pmatrix}, \quad \begin{pmatrix}
        0 & 1\\
        0 & 1\\
        0 & 0
    \end{pmatrix}, \quad \begin{pmatrix}
        0 & 0\\
        1 & 1\\
        1 & 0
    \end{pmatrix}.$$
\end{example}


\section{Remarks and open questions}\label{sec: final_remarks}

In this paper, we introduced and initiated the study of the polytope of all $q$-rank functions.
Here, we gather some remarks and future research directions. 

\begin{enumerate}
    \item[(a)] We highlight that the same study can be done for polymatroids. Moreover, most of our treatment could also be extended to $\mL$-polymatroids and more generally to latroids. In particular, all the arguments that depend only on the $q$-rank function and hold for a general lattice will follow in the same way.
    \item[(b)] The dimension of $\RP$ is $\sum\limits_{i=1}^n\qqbin{n}{i}_q$, making it even difficult to perform computations with the computer. Moreover, we observe that $\RP$
    contains all the $q$-rank functions of $q$-polymatroids. However, we are not able to distinguish isomorphic classes. This leads to the following question: Is it possible to find a way to identify isomorphic classes in order to reduce the dimension of $\RP$? 
    \item[(c)] Currently, we are not aware of methods for describing all the vertices of $\RP$. In particular, while the integer vertices correspond to $q$-matroids, the rational vertices are still to be identified. We leave this investigation for future work. 
    \item[(d)] By inspiration from \cite{grabisch2019cone} one may also study the polytope of all the nullity functions of a ($q$-)matroid. Indeed, this is non-negative, bounded, non-decreasing and supermodular.
    \item[(e)] In \Cref{sec: Puiseux_polyn}, we have already pointed out that the characteristic Puiseux polynomial of a rational $q$-polymatroid behaves similarly to the characteristic polynomial of a $(q,r)$-polymatroid, when it comes to the lattice-theoretic properties of this polynomial (\Cref{rem: char_Puiseux_series}(d)). Nevertheless, the characteristic polynomial of a $(q,r)$-polymatroid has many properties in terms of independent spaces, circuits, flats, etc. So future work could be to generalize these properties to the rational $q$-polymatroid counterparts.
    \item[(f)] In \Cref{sec: Puiseux_polyn} we initiated the study of the characteristic polynomial of a rational convex combination of $q$-matroids. Here, the result in \Cref{prop: char_Puiseux_series_paving_conv_comb} covers only a special situation. This leads to the following questions: Can we describe the characteristic Puiseux polynomial in other situations? Can we say something about its properties, in general? Is it possible to express it only in terms of the characteristic polynomials of the summands?   
\end{enumerate}

\section*{Acknowledgments}
The authors are very grateful to the editor and the reviewers for their close reading and many constructive comments, especially for the suggestion of adding \Cref{prop:duality}.
The authors are thankful to Lukas K\"uhne and Jan Stricker for fruitful discussions. Moreover, they are grateful to Bernd Sturmfels for inspiring this work. The second author would like to especially thank Lukas K\"uhne for his supervision and guidance throughout the research process. The first author is supported by the ANR through grant n. ANR-24-CPJ1-0075-01. The second author is supported by the Deutsche Forschungsgemeinschaft (DFG, German Research Foundation) -- SFB-TRR 358/1 2023 -- 491392403 and SPP 2458 -- 539866293.

\bigskip
\bigskip
\bigskip

\bibliographystyle{abbrv}
\bibliography{references.bib}
\end{document}